\documentclass{amsart}

\usepackage{enumerate, amsmath, amsfonts, amssymb, amsthm, thmtools, wasysym, graphics, graphicx, xcolor, url, hyperref, hypcap, a4wide, pdflscape, multido, xargs, colortbl, multicol, multirow, calc, shuffle, hvfloat}
\hypersetup{colorlinks=true, citecolor=PineGreen, linkcolor=PineGreen}
\usepackage{tikz}
\usetikzlibrary{trees, decorations, decorations.markings, shapes, arrows, matrix, calc, fit, intersections, patterns, angles, cd}
\usepackage{tikz-qtree}
\usepackage{comment}
\usepackage{etex}
\usepackage{ulem}
\usepackage[noabbrev,capitalise]{cleveref}
\graphicspath{{figures/}}
\makeatletter
\def\input@path{{figures/}}
\makeatother

\title{Interval hypergraphic lattices}

\author[N.~Bergeron]{Nantel Bergeron} 
\address[N.~Bergeron]{Department of Mathematics and Statistics, York University, Toronto}
\email{bergeron@yorku.ca}
\urladdr{http://bergeron.mathstats.yorku.ca}

\author{Vincent Pilaud}
\address[V.~Pilaud]{Universitat de Barcelona \& Centre de Recerca Matemàtica, Barcelona}
\email{vincent.pilaud@ub.edu}
\urladdr{https://www.ub.edu/comb/vincentpilaud/}

\thanks{
NB was supported by NSERC and York Research Chair in Applied Algebra.
VP was supported by the Spanish project PID2022-137283NB-C21 of MCIN/AEI/10.13039/501100011033 / FEDER, UE, by the Spanish--German project COMPOTE (AEI PCI2024-155081-2 \& DFG 541393733), by the Severo Ochoa and María de Maeztu Program for Centers and Units of Excellence in R\&D (CEX2020-001084-M), by the Departament de Recerca i Universitats de la Generalitat de Catalunya (2021 SGR 00697), and by the French--Austrian project PAGCAP (ANR-21-CE48-0020 \& FWF I 5788).}

\newtheorem{theorem}{Theorem}[section]
\newtheorem{theoremA}{Theorem}

\crefname{theoremA}{Theorem}{Theorems}
\newtheorem{corollary}[theorem]{Corollary}
\newtheorem{proposition}[theorem]{Proposition}
\newtheorem{lemma}[theorem]{Lemma}

\crefname{conjecture}{Conjecture}{Conjectures}

\crefname{conjectureA}{Conjecture}{Conjectures}
\crefname{conjectureA}{Conjecture}{Conjectures}

\theoremstyle{definition}
\newtheorem{definition}[theorem]{Definition}
\newtheorem{example}[theorem]{Example}
\newtheorem{remark}[theorem]{Remark}

\newtheorem{notation}[theorem]{Notation}

\crefname{notation}{Notation}{Notations}

\newcommand{\R}{\mathbb{R}} 
\renewcommand{\b}[1]{\boldsymbol{#1}} 
\newcommand{\cal}[1]{\mathcal{#1}} 

\newcommand{\set}[2]{\left\{ #1 \;\middle|\; #2 \right\}} 
\newcommand{\ssm}{\smallsetminus} 
\newcommand{\eqdef}{\mbox{\,\raisebox{0.2ex}{\scriptsize\ensuremath{\mathrm:}}\ensuremath{=}\,}} 
\newcommand{\defeq}{\mbox{~\ensuremath{=}\raisebox{0.2ex}{\scriptsize\ensuremath{\mathrm:}} }} 
\newcommand{\simplex}{\triangle} 

\newcommand{\ie}{\textit{i.e.}~} 
\newcommand{\aka}{\textit{aka.}~} 
\newcommand{\para}[1]{\medskip\noindent\textbf{#1}} 
\definecolor{PineGreen}{RGB}{2,120,120} 
\definecolor{darkgreen}{RGB}{57,181,74} 
\newcommand{\blue}[1]{{\color{blue} #1}} 
\newcommand{\green}[1]{{\color{darkgreen} #1}} 
\newcommand{\defn}[1]{\textbf{\textsf{\color{PineGreen} #1}}} 
\usepackage{todonotes}

\newcommand{\fS}{\mathfrak{S}} 

\newcommand{\meet}{\wedge} 
\newcommand{\join}{\vee} 
\newcommand{\bigJoin}{\bigvee} 
\newcommand{\less}{\vartriangleleft} 
\newcommand{\more}{\vartriangleright} 
\newcommand{\projDown}{\pi^\downarrow} 
\newcommand{\projUp}{\pi^\uparrow} 

\newcommand{\Or}{\mathcal O}  
\newcommand{\HH}{\mathbb H}  
\newcommand{\II}{\mathbb I} 
\newcommand{\cJ}{\cal{J}} 
\newcommand{\cIJ}{\cal{IJ}} 

\newcommand{\intervalHypergraph}[2]{
	\begin{tikzpicture}[baseline=0]
		\foreach \x in {1,...,#1} {
			\node (\x) at (\x*.5,-.3) [inner sep = -1pt] {$\scriptstyle \x$};
		}
		\newcount{\y} \y=0
		\foreach \a/\b in {#2} {
			\draw [thick,{Bar[width=3pt]}-{Bar[width=3pt]}] (\a*.5,\y*.2)--(\b*.5,\y*.2);
			\global\advance\y by 1
		}
		\node at (.5,0) {\phantom{$\bullet$}};
		\node at (#1*.5,0) {\phantom{$\bullet$}};
	\end{tikzpicture}
}

\newcommand{\acyclicOrientation}[2]{
	\begin{tikzpicture}[baseline=0]
		\foreach \x in {1,...,#1} {
			\node (\x) at (\x*.5,-.3) [inner sep = -1pt] {$\scriptstyle \x$};
		}
		\newcount{\y} \y=0
		\foreach \a/\b/\c in {#2} {
			\draw [thick,{Bar[width=3pt]}-{Bar[width=3pt]}] (\a*.5,\y*.2)--(\b*.5,\y*.2); \node at (\c*.5,\y*.2) {$\bullet$};
			\global\advance\y by 1
		}
		\node at (.5,0) {\phantom{$\bullet$}};
		\node at (#1*.5,0) {\phantom{$\bullet$}};
	\end{tikzpicture}
}

\newcommand{\tree}[1]{
	\begin{tikzpicture}[level distance = 20pt]
	\Tree #1
	\end{tikzpicture}
}

\newcommand{\flip}[4]{\ensuremath{#1 \, \begin{tikzpicture}[scale=1,baseline=-.1cm] \path[->]  (0,0) edge node[fill=white,inner sep=2pt, font=\tiny ] {$#2#3$} (.7,0); \end{tikzpicture} \, #4}}

\begin{document}

\begin{abstract}
For a hypergraph~$\HH$ on~$[n]$, the hypergraphic poset~$P_\HH$ is the transitive closure of the oriented skeleton of the hypergraphic polytope~$\simplex_\HH$ (the Minkowski sum of the standard simplices~$\simplex_H$ for all~$H \in \HH$).
Hypergraphic posets include the weak order for the permutahedron (when~$\HH$ is the complete graph on~$[n]$) and the Tamari lattice for the associahedron (when~$\HH$ is the set of all intervals of~$[n]$), which motivates the study of lattice properties of hypergraphic posets.
In this paper, we focus on interval hypergraphs, where all hyperedges are intervals of~$[n]$.
We characterize the interval hypergraphs~$\II$ for which~$P_\II$ is a lattice, a distributive lattice, a semidistributive lattice, and a lattice quotient of the weak order.
\end{abstract}

\maketitle

\tableofcontents


\pagebreak

\section{Introduction}
\label{sec:introduction}

Fix an integer~$n \ge 1$ and denote by~$(\b{e}_i)_{i \in [n]}$ the standard basis of~$\R^n$.
The \defn{hypergraphic polytope} of a hypergraph~$\HH$ on~$[n]$ is the Minkowski sum
\(
\simplex_\HH \eqdef \sum_{H\in \HH} \simplex_H\,,
\)
where $\simplex_H$ is the simplex given by the convex hull of the points $\b{e}_h$ for~$h \in H$.
These polytopes (or some special cases) were studied in~\cite{FeichtnerSturmfels, Postnikov, PostnikovReinerWilliams, AgnarssonMorris, Agnarsson, AguiarArdila, BenedettiBergeronMachacek, PadrolPilaudPoullot-deformationConesHypergraphicPolytopes, Rehberg, CardinalSteiner, CardinalHoangMerinoMickaMutze} among others.
In particular, the face lattice of~$\simplex_\HH$ was described combinatorially in terms of acyclic orientations of~$\HH$ in~\cite{BenedettiBergeronMachacek}.
Note that the singletons of~$\HH$ are irrelevant for the combinatorics of~$\simplex_\HH$ as they just contribute to translations.
It is convenient for us to assume that~$\{i\} \in \HH$ for all~$i \in [n]$.

The \defn{hypergraphic poset}~$P_\HH$ is the transitive closure of the skeleton of~$\simplex_\HH$ oriented in the direction~$\b{\omega} \eqdef (n, n-1, \dots, 2, 1) - (1, 2, \dots, n-1, n) = (n-1, n-3, \dots, 3-n, 1-n)$.
For instance, 
\begin{itemize}
\item if~$\HH = \binom{[n]}{2}$ is the complete graph (or any hypergraph containing it), then~$\simplex_\HH$ is the \defn{permutahedron} and $P_\HH$ is the \defn{weak order on permutations},
\item if~$\HH = \set{[i,j]}{1 \le i \le j \le n}$ is the complete interval hypergraph, then~$\simplex_\HH$ is \mbox{J.-L.~Loday's} \defn{associahedron}~\cite{ShniderSternberg,Loday} and~$P_\HH$ is the \defn{Tamari lattice on binary trees}~\cite{Tamari}.
\end{itemize}

Further important examples of hypergraphic polytopes include graphical zonotopes, graph associahedra~\cite{CarrDevadoss}, nestohedra~\cite{FeichtnerSturmfels,Postnikov}, multiplihedra~\cite{Stasheff-HSpaces, SaneblidzeUmble-diagonals, Forcey-multiplihedra, ArdilaDoker, ChapotonPilaud}, and constrainahedra~\cite{BottmanPoliakova, ChapotonPilaud}, and other examples given below.
Hypergraphic polytopes belong to the more general class of \defn{deformed permutahedra} (\aka \defn{generalized permutahedra}~\cite{Postnikov}, or \defn{polymatroids}~\cite{Edmonds}).
The later are all deformations of the permutahedron, \ie all polytopes whose normal fan coarsen the braid arrangement, and are all obtained as Minkowski sums and differences of faces of the standard simplex.
Among those, the hypergraphic polytopes are those which can be constructed using only Minkowski sums of faces of the standard simplex.

In view of the two above examples, we would like to characterize the hypergraphs~$\HH$ for which~$P_\HH$ is a lattice, a distributive lattice, a semidistributive lattice, a congruence-uniform lattice, a \mbox{(semi-)lattice} quotient of the weak order on permutations, etc.
These questions were settled in~\cite{Pilaud-acyclicReorientationLattices} for graphical zonotopes and in~\cite{ChapotonPilaud} for multiplihedra and constrainahedra.
In contrast, they were only partially answered in~\cite{BarnardMcConville} for graph associahedra~\cite{CarrDevadoss}, and still remain largely open for nestohedra~\cite{FeichtnerSturmfels, Postnikov}.

In this paper, we study the case of \defn{interval hypergraphs}~$\II$, \ie when all hyperedges of~$\II$ are intervals of~$[n]$.
Note that the family of interval hypergraphic polytopes includes
\begin{itemize}
\item the classical associahedron of~\cite{ShniderSternberg,Loday} when~$\II$ contains all intervals of~$[n]$,
\item the Pitman--Stanley polytope~\cite{PitmanStanley} when~$\II$ is the set of all singletons~$\{i\}$ and all initial intervals~$[i]$~for~${i \in [n]}$,
\item the freehedron of~\cite{Saneblidze-freehedron} when~$\II$ is the set of all singletons~$\{i\}$, all initial intervals~$[i]$ for~${i \in [n]}$, and all final intervals~$[n] \ssm [i]$~for~${i \in [n-1]}$,
\item the fertilitopes of~\cite{Defant-fertilitopes} when any two intervals of~$\II$ are either nested or disjoint.
\end{itemize}
In fact, it follows from~\cite{BazierMatteChapelierLaguetDouvilleMousavandThomasYildirim,PadrolPaluPilaudPlamondon} that the interval hypergraphic polytopes are precisely the weak Minkowski summands of the classical associahedron (recall that a polytope~$P \subset \R^n$ is a \defn{weak Minkowski summand} of a polytope~$Q \subset \R^n$ if there exists a real~$\lambda \ge 0$ and a polytope~$R \subset \R^d$ such that~$\lambda Q = P + R$).
In other words, interval hypergraphic polytopes are all deformations of the classical associahedron, \ie all polytopes whose normal fan coarsens the sylvester fan.
In contrast, the permutahedron, graphical zonotopes, graph associahedra, nestohedra, multiplihedra and constrainahedra are not interval hypergraphic polytopes.

We obtain the following characterizations, where we assume that~$\{i\} \in \II$ for all~$i \in [n]$ as mentioned earlier.
See \cref{sec:LatticeI,sec:distributive,sec:semidistributive,sec:quotient} and \cref{fig:notLattices,fig:Tamari,fig:distributiveLattices,fig:semidistributiveLattices,fig:notSemidistributiveLattices} for illustrations.

\begin{theoremA}
\label{thm:latticeI}
For an interval hypergraph $\II$, the poset $P_\II$ is a lattice if and only if $\II$ is closed under intersection (\ie $I, J \in \II$ and~$I \cap J \ne \varnothing$ implies~$I \cap J \in \II$).
\end{theoremA}

\begin{theoremA}
\label{thm:distributiveLatticeI}
For an interval hypergraph $\II$, the poset $P_\II$ is a distributive lattice if and only if for all~$I, J \in \II$ such that~$I \not\subseteq J$, $I \not\supseteq J$ and~$I \cap J \ne \varnothing$, the intersection~$I \cap J$ is in~$\II$ and is initial or final in any~$K \in \II$ with~$I \cap J \subseteq K$.
\end{theoremA}

\begin{theoremA}
\label{thm:semidistributiveLatticeI}
For an interval hypergraph $\II$, the poset $P_\II$ is a join semidistributive lattice if and only if~$\II$ is closed under intersection and for all~$[r,r'], [s,s'], [t,t'], [u,u'] \in \II$ such that ${r < s \le r' < s'}$, $r < t \le s' < t'$, $u < \min(s, t)$ and~$s' < u'$, there is~$[v,v'] \in \II$ such that~$v < s$ and~${s' < v' < t'}$.
A symmetric characterization holds for meet semidistributivity.
\end{theoremA}

\begin{theoremA}
\label{thm:quotientLatticeI}
For an interval hypergraph $\II$, the poset morphism from the weak order to the poset~$P_\II$ is a meet (resp.~join) semilattice morphism if and only if~$\II$ is closed under initial (resp.~final) subintervals (\ie~$[i,k] \in \II$ implies~$[i,j] \in \II$ (resp.~$[j,k] \in \II$) for any~$1 \le i < j < k \le n$).
\end{theoremA}

For instance, among the four above-mentioned families of interval hypergraphic polytopes, we recover that  the Pitman-Stanley polytope and all fertilitopes yield distributive lattices, the associahedron yields a semidistributive (but not distributive) lattice which is a quotient of the weak order, while the freehedron is not even a lattice (this was actually the motivation for~\cite{PilaudPoliakova} to construct alternative realizations of the skeleton of the freehedron).

Once \cref{thm:latticeI} is established, an important step for \cref{thm:distributiveLatticeI,thm:semidistributiveLatticeI} is to understand join irreducible elements of~$P_\II$.
\cref{subsec:joinIrreducibles} provides a combinatorial description of the join irreducible elements of the lattice~$P_\II$ for an arbitrary interval hypergraph~$\II$ closed under intersections.
To prepare this slightly technical description, we already describe some join irreducible elements of~$P_\II$ in \cref{subsec:someJoinIrreducibles}, which happen to be all join irreducible elements of~$P_\II$ under the condition of~\cref{thm:distributiveLatticeI}.

The paper is organized as follows.
In \cref{sec:HP}, we recall basic properties of hypergraphic polytopes, we define hypergraphic posets, and we recall the natural poset morphism from the weak order on permutations to the hypergraphic poset~$P_\II$.
In \cref{sec:IHP}, we develop specific properties of interval hypergraphic polytopes, in particular a simple characterization of their vertices and a global description of the relations in their hypergraphic posets.
In \cref{sec:LatticeI}, we characterize the interval hypergraphs~$\II$ for which the interval hypergraphic poset~$P_\II$ is a lattice, proving \cref{thm:latticeI}.
In \cref{sec:distributive}, we describe a family of join irreducible elements of~$P_\II$ and we characterize the interval hypergraphs~$\II$ for which~$P_\II$ is a distributive lattice, proving \cref{thm:distributiveLatticeI}.
In \cref{sec:semidistributive}, we describe all join irreducible elements of~$P_\II$ and we characterize the interval hypergraphs~$\II$ for which~$P_\II$ is a join (or meet) semidistributive lattice, proving \cref{thm:semidistributiveLatticeI}.
In \cref{sec:quotient}, we characterize the interval hypergraphs~$\II$ for which the poset morphism from the weak order on permutations to~$P_\II$ is a join (or meet) semilattice morphism, proving \cref{thm:quotientLatticeI}.

\setcounter{theoremA}{0}


\section{Hypergraphic posets}
\label{sec:HP}


\subsection{Hypergraphic polytopes}
\label{subsec:D_H}

A \defn{hypergraph} $\HH$ on $[n] \eqdef \{1, \dots, n\}$ is a collection of  subsets of~$[n]$.
By convention, we always assume that~$\HH$ contains all singletons~$\{i\}$ for~$i\in [n]$.
The \defn{hypergraphic polytope}~$\simplex_\HH$ is the Minkowski sum
\[
\simplex_\HH \eqdef \sum_{H\in \HH} \simplex_H\,,
\]
where $\simplex_H$ is the simplex given by the convex hull of the points $\b{e}_h \in \R^n$ for~$h \in H$.

\begin{example}
\label{exm:DH1}
For the hypergraph $\HH=\{ 1, 2, 3, 4, 123, 134 \}$,
we  have
\[
\begin{array}{ccc}
 \begin{tikzpicture}[scale=1,baseline=.5cm]
	\node (1) at (0.6,1.0) {$\scriptstyle e_1$};
	\node (2) at (-.2,-.2) {$\scriptstyle e_2$};
	\node (3) at (1.3,-.2) {$\scriptstyle e_3$};
	\draw [fill=blue!40] (0,0) -- (.6,.75) -- (1.2,0) --(0,0) ;   
\end{tikzpicture} \quad &
 \begin{tikzpicture}[scale=1,baseline=.5cm]
	\node (1) at (0.6,1.0) {$\scriptstyle e_1$};
	\node at (1.3,-.2) {$\scriptstyle e_3$};
	\node at (2.4,.5) {$\scriptstyle e_4$};
	\draw [fill=green!20] (1.2,0) -- (2.2,.5 )--(.6,.75)--(1.2,0); 
\end{tikzpicture} \quad &
\begin{tikzpicture}[scale=1,baseline=.5cm]
	\draw [color=gray!10,fill=blue!2] (1.6,1.25)--(1,.5)--(2.2,.5)-- (1.6,1.25); 
	\draw [color=gray!10,fill=green!5]  (0,0)--(1,.5)--(-.6,.75)--(0,0);
	\draw [dotted,color=green] (0,0)--(1,.5)--(-.6,.75);
	\draw [dotted,color=blue] (1.6,1.25)--(1,.5)--(2.2,.5);
	\draw (0,1.5)-- (1.6,1.25) -- (2.2,.5) -- (1.2,0) --(0,1.5) ; 
	\draw (0,0)--(-.6,.75)--(0,1.5)-- (1.2,0) --(0,0); 
\end{tikzpicture}\\
\blue{\simplex_{123}}& \green{\simplex_{134}} & \simplex_\HH = \simplex_1 + \simplex_2 + \simplex_3 + \simplex_4 + \blue{\simplex_{123}} + \green{\simplex_{134}}\\
\end{array}
\]
which is a 3-dimensional polytope sitting in $\R^4$.
Note that as~$n \le 9$ in all our examples, we simplify notations and write~$123$ for the set~$\{1,2,3\}$.
\end{example}

\begin{remark}
\label{rem:single}
\enlargethispage{.2cm}
Note that the singleton hyperedges are irrelevant for our purposes.
Namely, adding to~$\HH$ the hyperedge~$\{i\}$ for some~$i \in [n]$ just translates the polytope~$\simplex_\HH$ in the direction~$\b{e}_i$, which does not affect the face structure of the polytope.
For our conditions on hypergraphs of \cref{thm:latticeI,thm:distributiveLatticeI}, it is convenient for us to assume that $\{i\} \in \HH$ for all $i \in [n]$.
When quoting results from~\cite{BenedettiBergeronMachacek}, the reader will have to be mindful that they took the opposite convention.
\end{remark}


\subsection{Acyclic orientations, increasing flips, and hypergraphic posets} 
\label{subsec:P_H}

We now recall from \cite[Thm.~2.18]{BenedettiBergeronMachacek} a combinatorial model for the graph~$(V_\HH, E_\HH)$ of~$\simplex_\HH$.
As we are only interested in the $1$-skeleton of~$\simplex_\HH$, we simplify some of the general definitions of~\cite{BenedettiBergeronMachacek} which were designed to deal with all faces of~$\simplex_\HH$.

\begin{definition}
\label{def:acyclicOrientation}
An \defn{orientation} of~$\HH$ is a map~$O$ from~$\HH$ to~$[n]$ such that~$O(H) \in H$ for all~${H \in \HH}$.
Equivalently, we often represent the orientation~$O$ as the set of pairs~$\set{(O(H),H)}{H \in \HH}$.
The orientation~$O$ is \defn{acyclic} if there is no~$H_1, \dots, H_k$ with~$k \ge 2$ such that~$O(H_{i+1}) \in H_i \ssm \{O(H_i)\}$ for~$i \in [k-1]$ and~$O(H_1) \in H_k \ssm \{O(H_k)\}$.
\end{definition}

\begin{remark}
The singleton~$\{i\}$ of~$\HH$ is always oriented $O(\{i\})=i$, and plays no role in determining  the acyclicity of~$O$ since $\{i\} \ssm \{O(\{i\})\} = \varnothing$.
We will therefore omit the singletons when describing or drawing orientations (see also~\cref{exm:intervalHypergraph}).
If we list the non-singleton elements $(H_1,H_2,...,H_k)$ of $\HH$ in some fixed order, it is then convenient to describe an orientation~$O$ as the tuple $O=(O(H_1),O(H_2),\ldots,O(H_k))$. 
\end{remark}

\begin{example}
\label{exm:DH2}
Using $\HH=\{ 1, 2, 3, 4, 123, 134 \}$ as in~\cref{exm:DH1}, we order the two non-singleton $(123,134)$.
There are $9$ orientations of~$\HH$, $7$ of which are acyclic as displayed in~\cref{fig:Orientation1}. 
For instance, the orientation $(O(123), O(134)) = (1,3)$ is cyclic since $O(123) = 1 \in 134$ and ${O(134) = 3 \in 123}$ is a cycle with $k = 2$.
\begin{figure}
	\centerline{
	\begin{tabular}{c@{\qquad}c}
		\begin{tikzpicture}[scale=2,baseline=.5cm]
		\draw [color=gray!10,fill=blue!2] (1.6,1.25)--(1,.5)--(2.2,.5)-- (1.6,1.25); 
		\draw [color=gray!10,fill=green!5]  (0,0)--(1,.5)--(-.6,.75)--(0,0);
		\node (23) at (0,0) {$\scriptstyle (2,3)$}; 		
		\node (21) at (-.6,.75) {$\scriptstyle (2,1)$};	
		\node (24) at (1,.5) {$\scriptstyle (2,4)$}; 	
		\node (34) at (2.2,.5) {$\scriptstyle (3,4)$};
		\node (14) at (1.6,1.25) {$\scriptstyle (1,4)$};	
		\node (33) at (1.2,0) {$\scriptstyle (3,3)$};	
		\node (11) at (0,1.5) {$\scriptstyle (1,1)$};		
		\path[color=red, ->] (11) edge node[fill=white,inner sep=2pt, font=\tiny ] {$  12 $} (21);
		\path[color=red, ->] (11) edge node[fill=white,inner sep=2pt, font=\tiny ] {$  13 $} (33);
		\path[color=red, ->] (11) edge node[fill=white,inner sep=2pt, font=\tiny ] {$  14 $} (14);
		\path[color=red, ->] (14) edge node[fill=white,inner sep=2pt, font=\tiny ] {$  13 $} (34) ;
		\path[dotted,color=red, ->] (14) edge node[fill=white,inner sep=2pt, font=\tiny ] {$  12 $} (24)  ;
		\path[dotted,color=red, ->] (24) edge node[fill=white,inner sep=2pt, font=\tiny ] {$  23 $} (34)  ;
		\path[dotted,color=red, ->] (21) edge node[fill=white,inner sep=2pt, font=\tiny ] {$  14$} (24);
		\path[color=red, ->] (21) edge node[fill=white,inner sep=2pt, font=\tiny ] {$  13$} (23);
		\path[dotted,color=red, <-]  (24) edge node[fill=white,inner sep=2pt, font=\tiny ] {$  34$} (23);
		\path[color=red, ->]  (23) edge node[fill=white,inner sep=2pt, font=\tiny ] {$  23$} (33);
		\path[color=red, ->]  (33) edge node[fill=white,inner sep=2pt, font=\tiny ] {$  34$} (34);
		\end{tikzpicture}
	&
		\begin{tikzpicture}[scale=1,baseline=.5cm]
		\node (11) at (0,0) {$\scriptstyle (1,1)$};
		\node (21) at (-1,1) {$\scriptstyle (2,1)$};	
		\node (23) at (-1,2) {$\scriptstyle (2,3)$}; 		
		\node (14) at (1,1.5) {$\scriptstyle (1,4)$};	
		\node (24) at (0,3) {$\scriptstyle (2,4)$}; 	
		\node (33) at (-2,3) {$\scriptstyle (3,3)$};
		\node (34) at (-1,4) {$\scriptstyle (3,4)$};	
		\draw (11)--(21);
		\draw (11)--(14);
		\draw (21)--(23);
		\draw (23)--(24);
		\draw (14)--(24);
		\draw (23)--(33);
		\draw (33)--(34);
		\draw (24)--(34);
		\end{tikzpicture}
	\\
	$\Delta_\HH$ & $P_\HH$\\
	\end{tabular}
	}
	\caption{The polytope $\Delta_{\HH}$ for $\HH=\{ 1, 2, 3, 4, 123, 134 \}$ has seven vertices corresponding to the acyclic orientations of $\HH$ and eleven (oriented) edges corresponding to the increasing flips between these orientations.
	The poset $P_\HH$ is the transitive closure of the increasing flip graph.}
	\label{fig:Orientation1}
\end{figure}
\end{example}

\begin{definition}
\label{def:flip}
Two orientations~$O \ne O'$ of~$\HH$ are related by an \defn{increasing flip} if there exist~${1 \le i < j \le n}$ such that for all~$H \in \HH$, 
\begin{itemize}
\item if~$O(H) \ne O'(H)$, then~$O(H) = i$ and~$O'(H) = j$, and
\item if~$\{i,j\} \subseteq H$, then~$O(H) = i \iff O'(H) = j$.
\end{itemize}
We denote such a flip by~\flip{O}{i}{j}{O'}.
\end{definition}

\begin{example}
\label{exm:DH3}
\cref{fig:Orientation1} shows all the increasing flips between the acyclic orientations of the hypergraph $\HH=\{ 1, 2, 3, 4, 123, 134 \}$ of \cref{exm:DH1,exm:DH2}.
For instance, \flip{(2,3)}{2}{3}{(3,3)} indicates that there is an increasing flip from  $O=(2,3)$ to $O'=(3,3)$ with $i=2<3=j$.
\end{example}

The following correspondance was already observed in~\cite[Thm.~2.18]{BenedettiBergeronMachacek} (it even extends to all faces of~$\simplex_\HH$, but we do not need this level of generality in this paper).
We provide an alternative short proof for convenience.

\begin{proposition}[{\cite[Thm.~2.18]{BenedettiBergeronMachacek}}]
\label{prop:Hgraph}
The graph of the hypergraphic polytope~$\simplex_\HH$ oriented in the direction~$\b{\omega} \eqdef (n, n-1, \dots, 2, 1) - (1, 2, \dots, n-1, n) = (n-3, n-1, \dots, 1-n, 3-n)$ is isomorphic to the increasing flip graph on acyclic orientations of~$\HH$.
\end{proposition}

\begin{proof}
Recall that the face of a Minkowski sum~$\sum_i P_i$ minimizing a direction~$\b{v}$ is the Minkowski sum of the faces of the summands~$P_i$ minimizing~$\b{v}$.

The vertex of~$\simplex_H$ minimizing a generic direction~$\b{v}$ is~$\b{e}_i$ for~$i \in H$ such that~$\b{v}_i = \min\set{\b{v}_h\!}{\!h \!\in\! H}$.
An acyclic orientation of~$\HH$ corresponds to the choice of one vertex in each~$\simplex_H$, and the orientation is acyclic if and only if this choice corresponds to a generic orientation~$\b{v}$, hence to a vertex of~$\simplex_\HH$.

The edges of~$\simplex_H$ are oriented in the directions~$\b{e}_i-\b{e}_j$ for~$i,j \in H$.
The edges of~$\simplex_\HH$ are thus also oriented by $\b{e}_i-\b{e}_j$, and thus correspond to pairs of acyclic orientations which differ by~a~flip.
\end{proof}

Finally, the main objects of this paper are the following posets.

\begin{definition}
The \defn{hypergraphic poset}~$P_\HH$ is the transitive closure of the increasing flip graph on acyclic orientations of~$\HH$.
\end{definition}

\begin{example}
\label{exm:DH4}
For the hypergraph~$\HH=\{ 1, 2, 3, 4, 123, 134 \}$ of \cref{exm:DH1,exm:DH2,exm:DH3}, the poset  $P_\HH$ is represented in~\cref{fig:Orientation1}\,(right).
\end{example}

\begin{remark}
\label{rem:edgeNotCover}
As seen in~\cref{fig:Orientation1} the edges in $E_\HH$ are not necessarily cover relations in~$P_\HH$.
The simplest case is given by
$\HH=\{ 1, 2, 3, 123 \}$, 
whose hypergraphic polytope~is
\[
	\simplex_\HH = \simplex_1 + \simplex_2 + \simplex_3 + \simplex_{123} =
	\begin{tikzpicture}[scale=.7,baseline=.0cm]
		\draw [color=gray!10,fill=blue!2] (0,-1)--(0,1)--(1.632,0)--(0,-1) ; 
		\node (a) at (0,-1) {$\scriptstyle (1)$};
		\node (b) at (1.632,0) {$\scriptstyle (2)$};
		\node (c) at (0,1) {$\scriptstyle (3)$};
		\path[color=red, ->] (a) edge node[fill=white,inner sep=1pt, font=\tiny ] {$ \scriptscriptstyle 13 $} (c);
		\path[color=red, ->] (a) edge node[fill=white,inner sep=1pt, font=\tiny ] {$\scriptscriptstyle  12 $} (b);
		\path[color=red, ->] (b) edge node[fill=white,inner sep=1pt, font=\tiny ] {$ \scriptscriptstyle 23 $} (c);
	\end{tikzpicture}
\]
and the left edge~$\flip{(1)}{1}{3}{(3)}$ of $\simplex_\HH$ is not a cover relation of~$P_\HH$.
\end{remark}

We conclude with an elementary yet relevant symmetry.

\begin{proposition}
\label{prop:antiIsomorphism}
For~$x \in [n]$, define~$x^\leftrightarrow \eqdef n-x+1$.
For~$H \subseteq [n]$, define~$H^\leftrightarrow \eqdef \set{h^\leftrightarrow}{h \in H}$.
For an hypergraph~$\HH$, define~$\HH^\leftrightarrow \eqdef \set{H^\leftrightarrow}{H \in \HH}$.
For an orientation~$O$ of~$\HH$, define the orientation~$O^\leftrightarrow$ of~$\HH^\leftrightarrow$ by~$O^\leftrightarrow(H^\leftrightarrow) \eqdef O(H)^\leftrightarrow$.
Then the map~$A \mapsto A^\leftrightarrow$ is an anti-isomorphism from~$P_\HH$ to~$P_{\HH^\leftrightarrow}$.
\end{proposition}

\begin{proof}
Straightforward.
\end{proof}


\subsection{Surjection map} 
\label{subsec:surjection}

The hypergraphic polytope~$\simplex_\HH$ is a deformed permutahedron (\aka generalized permutahedron~\cite{Postnikov, PostnikovReinerWilliams}, \aka polymatroid~\cite{Edmonds}).
This means that the normal fan of~$\simplex$ coarsens the normal fan of the permutahedron.
Hence, there is a natural surjection from the faces of the permutahedron to the faces of~$\simplex_\HH$, which was described in detail in~\cite[Lem.~2.9]{BenedettiBergeronMachacek}.
Here, we focus on the surjection~$\Or$ from the permutations of $[n]$ to the acyclic orientations of~$\HH$.

\begin{definition}
\label{def:surjection}
For a permutation~$\pi$ of~$[n]$, the orientation~$\Or_\pi$ of~$\HH$ is defined for all~$H \in \HH$ by
\[
\Or_\pi(H) \eqdef  \pi\big(\min\set{j}{\pi(j)\in H}\big).
\]
\end{definition}

\begin{proposition}[{\cite[Lem.~2.9]{BenedettiBergeronMachacek}}]
For any hypergraph~$\HH$ on~$[n]$,
\begin{itemize}
\item the map~$\Or$ is a surjection from the permutations of~$[n]$ to the acyclic orientations of~$\HH$,
\item two acyclic orientations~$A,B$ of~$\HH$ are related by a flip if and only if there are permutations~$\pi_A, \pi_B$ of~$[n]$ which differ by a simple transposition such that~$\Or_{\pi_A} = A$ and~$\Or_{\pi_B} = B$.
\end{itemize}
In other words, the graph~$(V_\HH, E_\HH)$ of~$\simplex_\HH$ is isomorphic to the graph obtained by contracting the fibers of~$\Or$ in the graph of the permutahedron.
\end{proposition}

\begin{corollary}
\label{coro:weakToP}
The map~$\Or$ defines a poset morphism from the weak order on permutations to the hypergraphic poset~$P_\HH$.
\end{corollary}

\pagebreak
Finally, we describe the fibers of the surjection~$\Or : \fS_n \to V_\HH$.
Given an acyclic orientation~$A$ of~$\HH$, define $\less_A$ as the order on $[n]$ obtained by the transitive closure of the union of the \linebreak orders $\set{A(H) < h}{h \in H \ssm \{A(H)\}}$ for each $H \in \HH$.
That is,
\[ 
	{\less_A} =  \text{Trans. cl. }
	\bigcup_{H \in \HH} 
	\begin{tikzpicture}[scale=1,baseline=.0cm]
		\node at (0,-.45) {$\scriptstyle A(H)$};
		\node at (0,.6) {$\scriptstyle h \in H \ssm \{A(H)\}$};
		\node at (.2,.35) {$\ldots$};
		\draw [thick] (0,-.3)--(-.5,.4); 
		\draw [thick] (0,-.3)--(-.3,.4); 
		\draw [thick] (0,-.3)--(-.1,.4); 
		\draw [thick] (0,-.3)--(.5,.4); 
	\end{tikzpicture} \,.
\]
This is a well defined partial order since $A$ is acyclic.
The following lemma is straightforward.

\begin{lemma}
\label{lem:prepi}
For any acyclic orientation $A$ of~$\HH$, the preimage $\Or^{-1}(A) \eqdef \set{\pi}{\Or_\pi=A}$ is the set of linear extensions of~$\less_A$.
\end{lemma}

\begin{example}
\label{exm:DH5}
Continuing with the hypergraph $\HH=\{ 1, 2, 3, 4, 123, 134 \}$ of \cref{exm:DH1,exm:DH2,exm:DH3,exm:DH4}, let $A$ be the acyclic orientation $(2,4)$.
The order~$\less_A$ is the following
\[
	{\less_A} = \text{Transitive closure }\Bigg(
	\begin{tikzpicture}[scale=1,baseline=.2cm]
		\node (2) at (0,0) {$\scriptstyle 2$};
		\node (1) at (-.4,.7) {$\scriptstyle 1$};
		\node (3) at (.4,.7) {$\scriptstyle 3$};
		\draw [thick] (2)--(1); 
		\draw [thick] (2)--(3); 
	\end{tikzpicture} 
	\cup
	\begin{tikzpicture}[scale=1,baseline=.2cm]
		\node (4) at (0,0) {$\scriptstyle 4$};
		\node (1) at (-.4,.7) {$\scriptstyle 1$};
		\node (3) at (.4,.7) {$\scriptstyle 3$};
		\draw [thick] (4)--(1); 
		\draw [thick] (4)--(3); 
	\end{tikzpicture} 
	\Bigg)
	=
	\begin{tikzpicture}[scale=1,baseline=.2cm]
		\node (4) at (.4,0) {$\scriptstyle 4$};
		\node (2) at (-.4,0) {$\scriptstyle 2$};
		\node (1) at (-.4,.7) {$\scriptstyle 1$};
		\node (3) at (.4,.7) {$\scriptstyle 3$};
		\draw [thick] (4)--(1); 
		\draw [thick] (4)--(3); 
		\draw [thick] (2)--(1); 
		\draw [thick] (2)--(3); 
	\end{tikzpicture} 	
\]
The linear extensions of this order are the permutations $2413$, $2431$, $4213$ and $4231$.
All other fibers are represented in \cref{fig:surjectionMap}.

\begin{figure}
	\centerline{
		\includegraphics[scale=.6]{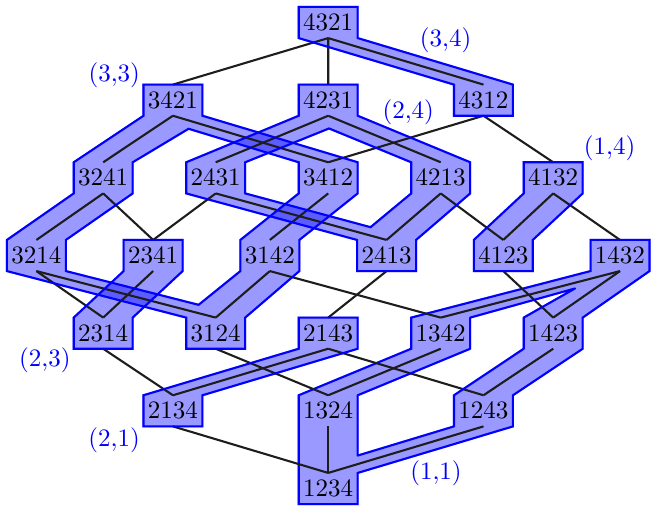}
		\qquad
		\raisebox{1cm}{
		\begin{tikzpicture}[scale=1,baseline=.5cm]
		\node (11) at (0,0) {$\scriptstyle (1,1)$};
		\node (21) at (-1,1) {$\scriptstyle (2,1)$};	
		\node (23) at (-1,2) {$\scriptstyle (2,3)$}; 		
		\node (14) at (1,1.5) {$\scriptstyle (1,4)$};	
		\node (24) at (0,3) {$\scriptstyle (2,4)$}; 	
		\node (33) at (-2,3) {$\scriptstyle (3,3)$};
		\node (34) at (-1,4) {$\scriptstyle (3,4)$};	
		\draw (11)--(21);
		\draw (11)--(14);
		\draw (21)--(23);
		\draw (23)--(24);
		\draw (14)--(24);
		\draw (23)--(33);
		\draw (33)--(34);
		\draw (24)--(34);
		\end{tikzpicture}
		}
	}
	\caption{The poset morphism~$\Or$ from the weak order (left) to the hypergraphical poset~$P_\HH$ (right), for $\HH=\{ 1, 2, 3, 4, 123, 134 \}$. The fibers of~$\Or$ appear as blue bubbles on the weak order, labeled by their acyclic orientation of~$\HH$ (left).}
	\label{fig:surjectionMap}
\end{figure}
\end{example}


\section{Interval hypergraphic posets}
\label{sec:IHP}

In this paper, we focus on the following family of hypergraphs on~$[n]$.

\begin{definition}
An \defn{interval hypergraph}~$\II$ is an hypergraph on~$[n]$ where each~$I \in \II$ is an interval of the form $I = [i,j] \eqdef \{i, i+1, i+2, \dots, j-1, j\}$.
\end{definition}

\begin{example}
Our running example~$\HH=\{ 1, 2, 3, 4, 123, 134 \}$  of \cref{exm:DH1,exm:DH2,exm:DH3,exm:DH4,exm:DH5}, is not an interval hypergraph, as $134$ is not an interval.
\end{example}

\begin{example}
\label{exm:intervalHypergraph}
The hypergraph
$\II = \{ 1, 2, 3, 4,123, 23, 234, 1234 \}$ 
is an interval hypergraph and
$\Or_{4132} = (1, 4, 3, 4)$ is an acyclic orientation of~$\II$ (for the lexicographic order $(123, 1234, 23, 234)$ of the non-singletons).
We will represent interval hypergraphs and their orientations graphically as follows:
\[
	\II =  \raisebox{-.2cm}{\intervalHypergraph{4}{1/3,1/4,2/3,2/4}}
	\qquad\qquad
	\Or_{4132}  =  \raisebox{-.2cm}{\acyclicOrientation{4}{1/3/1,1/4/4,2/3/3,2/4/4}}
\]
As before, we omit to draw the singletons $\{i\}$ for $i \in [n]$.
See also~\cref{fig:exmInterval}.
\begin{figure}
	\centerline{
		\begin{tabular}[t]{c}
			\begin{tikzpicture}%
	[x={(0.366215cm, 0.789554cm)},
	y={(-0.235950cm, 0.590693cm)},
	z={(-0.900119cm, 0.166391cm)},
	scale=1.300000,
	back/.style={very thin, opacity=0.5},
	edge/.style={color=blue, thick, decoration={markings, mark=at position 0.5 with {\arrow{>}}}},
	facet/.style={fill=blue,fill opacity=0},
	vertex/.style={}]

\coordinate (-1.33333, 0.00000, -1.88562) at (-1.33333, 0.00000, -1.88562);
\coordinate (-1.33333, 3.26599, -1.88562) at (-1.33333, 3.26599, -1.88562);
\coordinate (-1.33333, -1.63299, 0.94281) at (-1.33333, -1.63299, 0.94281);
\coordinate (-1.33333, 0.00000, 3.77124) at (-1.33333, 0.00000, 3.77124);
\coordinate (0.00000, -0.81650, -1.41421) at (0.00000, -0.81650, -1.41421);
\coordinate (0.00000, -1.63299, 0.00000) at (0.00000, -1.63299, 0.00000);
\coordinate (1.33333, 0.00000, -0.94281) at (1.33333, 0.00000, -0.94281);
\coordinate (1.33333, 1.63299, -0.94281) at (1.33333, 1.63299, -0.94281);
\coordinate (1.33333, -0.81650, 0.47140) at (1.33333, -0.81650, 0.47140);
\coordinate (1.33333, 0.00000, 1.88562) at (1.33333, 0.00000, 1.88562);
\draw[edge,postaction={decorate},back] (-1.33333, 0.00000, -1.88562) -- (-1.33333, 3.26599, -1.88562);
\draw[edge,postaction={decorate},back] (-1.33333, 3.26599, -1.88562) -- (-1.33333, 0.00000, 3.77124);
\draw[edge,postaction={decorate},back] (-1.33333, 3.26599, -1.88562) -- (1.33333, 1.63299, -0.94281);
\node[vertex] at (-1.33333, 3.26599, -1.88562)     {};
\fill[facet] (1.33333, 0.00000, 1.88562) -- (-1.33333, 0.00000, 3.77124) -- (-1.33333, -1.63299, 0.94281) -- (0.00000, -1.63299, 0.00000) -- (1.33333, -0.81650, 0.47140) -- cycle {};
\fill[facet] (0.00000, -1.63299, 0.00000) -- (-1.33333, -1.63299, 0.94281) -- (-1.33333, 0.00000, -1.88562) -- (0.00000, -0.81650, -1.41421) -- cycle {};
\fill[facet] (1.33333, -0.81650, 0.47140) -- (0.00000, -1.63299, 0.00000) -- (0.00000, -0.81650, -1.41421) -- (1.33333, 0.00000, -0.94281) -- cycle {};
\fill[facet] (1.33333, 0.00000, 1.88562) -- (1.33333, 1.63299, -0.94281) -- (1.33333, 0.00000, -0.94281) -- (1.33333, -0.81650, 0.47140) -- cycle {};
\draw[edge,postaction={decorate}] (-1.33333, 0.00000, -1.88562) -- (-1.33333, -1.63299, 0.94281);
\draw[edge,postaction={decorate}] (-1.33333, 0.00000, -1.88562) -- (0.00000, -0.81650, -1.41421);
\draw[edge,postaction={decorate}] (-1.33333, -1.63299, 0.94281) -- (-1.33333, 0.00000, 3.77124);
\draw[edge,postaction={decorate}] (-1.33333, -1.63299, 0.94281) -- (0.00000, -1.63299, 0.00000);
\draw[edge,postaction={decorate}] (-1.33333, 0.00000, 3.77124) -- (1.33333, 0.00000, 1.88562);
\draw[edge,postaction={decorate}] (0.00000, -0.81650, -1.41421) -- (0.00000, -1.63299, 0.00000);
\draw[edge,postaction={decorate}] (0.00000, -0.81650, -1.41421) -- (1.33333, 0.00000, -0.94281);
\draw[edge,postaction={decorate}] (0.00000, -1.63299, 0.00000) -- (1.33333, -0.81650, 0.47140);
\draw[edge,postaction={decorate}] (1.33333, 0.00000, -0.94281) -- (1.33333, 1.63299, -0.94281);
\draw[edge,postaction={decorate}] (1.33333, 0.00000, -0.94281) -- (1.33333, -0.81650, 0.47140);
\draw[edge,postaction={decorate}] (1.33333, 1.63299, -0.94281) -- (1.33333, 0.00000, 1.88562);
\draw[edge,postaction={decorate}] (1.33333, -0.81650, 0.47140) -- (1.33333, 0.00000, 1.88562);
\node[vertex] at (-1.33333, 0.00000, -1.88562)     {};
\node[vertex] at (-1.33333, -1.63299, 0.94281)     {};
\node[vertex] at (-1.33333, 0.00000, 3.77124)     {};
\node[vertex] at (0.00000, -0.81650, -1.41421)     {};
\node[vertex] at (0.00000, -1.63299, 0.00000)     {};
\node[vertex] at (1.33333, 0.00000, -0.94281)     {};
\node[vertex] at (1.33333, 1.63299, -0.94281)     {};
\node[vertex] at (1.33333, -0.81650, 0.47140)     {};
\node[vertex] at (1.33333, 0.00000, 1.88562)     {};
\end{tikzpicture} \\[.2cm]
			\includegraphics[scale=.7]{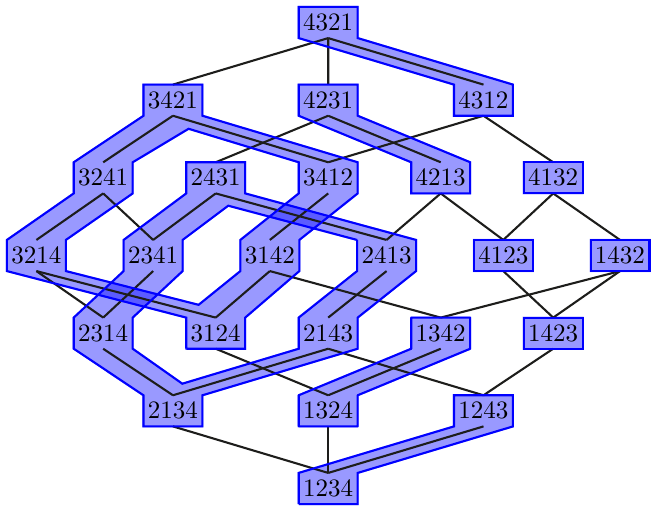}
		\end{tabular}
		\raisebox{-6.3cm}{
			\begin{tikzpicture}[scale=2.5]
				\node (a) at (3,0) {\acyclicOrientation{4}{1/3/1,1/4/1,2/3/2,2/4/2}};
				\node (b) at (1,1) {\acyclicOrientation{4}{1/3/2,1/4/2,2/3/2,2/4/2}};
				\node (c) at (3,1) {\acyclicOrientation{4}{1/3/1,1/4/1,2/3/3,2/4/3}};
				\node (d) at (4,1) {\acyclicOrientation{4}{1/3/1,1/4/1,2/3/2,2/4/4}};
				\node (e) at (1,2) {\acyclicOrientation{4}{1/3/3,1/4/3,2/3/3,2/4/3}};
				\node (f) at (3,2) {\acyclicOrientation{4}{1/3/1,1/4/4,2/3/2,2/4/4}};
				\node (g) at (4,2) {\acyclicOrientation{4}{1/3/1,1/4/1,2/3/3,2/4/4}};
				\node (h) at (2,3) {\acyclicOrientation{4}{1/3/2,1/4/4,2/3/2,2/4/4}};
				\node (i) at (3,3) {\acyclicOrientation{4}{1/3/1,1/4/4,2/3/3,2/4/4}};
				\node (j) at (2,4) {\acyclicOrientation{4}{1/3/3,1/4/4,2/3/3,2/4/4}};
				\draw (a)--(b);
				\draw (a)--(c);
				\draw (a)--(d);
				\draw (b)--(e);
				\draw (b)--(h);
				\draw (c)--(e);
				\draw (c)--(g);
				\draw (d)--(f);
				\draw (d)--(g);
				\draw (e)--(j);
				\draw (f)--(h);
				\draw (f)--(i);
				\draw (g)--(i);
				\draw (h)--(j);
				\draw (i)--(j);
			\end{tikzpicture}
		}
	}
	\vspace{-.5cm}
	\centerline{$\{ 1, 2, 3, 4,123, 23, 234, 1234 \}$}
	\caption{The interval hypergraphical polytope~$\simplex_\II$ (top left), the fibers of the corresponding map~$\Or$ (bottom left), and the interval hypergraphical poset~$P_\II$ (right), for~$\II = \{ 1, 2, 3, 4,123, 23, 234, 1234 \}$.}
	\label{fig:exmInterval}
\end{figure}
\end{example}

\begin{example}
We have represented further interval hypergraphic posets in \cref{fig:notLattices,fig:Tamari,fig:distributiveLattices,fig:semidistributiveLattices,fig:notSemidistributiveLattices}.
These figures illustrate in particular the following relevant families of interval hypergraphic posets mentioned in the introduction:
\begin{itemize}
\item if~$\II = \set{[i,j]}{1 \le i \le j \le n}$ is the set of all intervals of~$[n]$, then~$P_\II$ is the Tamari lattice corresponding to the classical associahedron of~\cite{ShniderSternberg,Loday}; see \cref{fig:Tamari};
\item if~$\II = \set{[1,i]}{i \in [n]}$ is the set of initial intervals of~$[n]$, then~$P_\II$ is the boolean lattice corresponding to the Pitman--Stanley polytope~\cite{PitmanStanley}; see \cref{fig:distributiveLattices}\,(bottom~right);
\item if~$\II = \set{[1,i]}{i \in [n]} \cup \set{[i,n]}{i \in [n]}$ is the set of all initial or final intervals of~$[n]$, then~$P_\II$ is a poset given by the freehedron of~\cite{Saneblidze-freehedron}, which is not even a lattice; see \cref{fig:notLattices}\,(right);
\item if any two intervals of~$\II$ are either nested or disjoint, then~$P_\II$ is a distributive lattice given by the fertilitope of~\cite{Defant-fertilitopes}; see \cref{fig:distributiveLattices}\,(top right and bottom) and \cref{subsec:SchroderHypergraphs}.
\end{itemize}
\end{example}

\begin{remark}
Recall that a polytope~$P \subset \R^n$ is a \defn{weak Minkowski summand} of a polytope~${Q \subset \R^n}$ if there exists a real~$\lambda \ge 0$ and a polytope~$R \subset \R^d$ such that~$\lambda Q = P + R$.
Equivalently, $P$ is a weak Minkowski summand of~$Q$ if the normal fan of~$Q$ refines the normal fan of~$P$.
The weak Minkowski summands of~$Q$ form a cone under Minkowski sum and dilation~\cite{McMullen-typeCone}, called the \defn{deformation cone} of~$Q$.
For instance, the deformation cone of the permutahedron is the cone of generalized permutahedra~\cite{Postnikov, PostnikovReinerWilliams} (also known as submodular cone, or cone of \defn{polymatroids}~\cite{Edmonds}).
Note that all hypergraphic polytopes are generalized permutahedra.
It follows from~\cite{BazierMatteChapelierLaguetDouvilleMousavandThomasYildirim,PadrolPaluPilaudPlamondon} that the deformation cone of the associahedron of~\cite{ShniderSternberg,Loday} is a simplicial cone generated by the faces of the standard simplex corresponding to intervals of~$[n]$.
In particular, the interval hypergraphic polytopes are precisely the weak Minkowski summands of the associahedron~\cite{ShniderSternberg,Loday}.
\end{remark}


\subsection{Acyclic orientations for interval hypergraphs}
\label{subsec:acyclicI}

In this section we give a simple characterization of the acyclic orientations of an interval hypergraph~$\II$.

\begin{proposition}
\label{prop:acyclicI}
An orientation~$O$ of an interval hypergraph~$\II$ is acyclic if and only if there is no~$I,J \in \II$ such that~$O(I) \in J \ssm \{O(J)\}$ and~$O(J) \in I \ssm \{O(I)\}$.
Graphically, there is no pattern
 \[
	\begin{tikzpicture}[scale=1,baseline=.0cm]
	\node at (-.5,0.2) {$\scriptscriptstyle J$}; 
	\node at (-.1,0.2) {$\scriptstyle \cdot$};
	\node at (0,0.2) {$\scriptstyle \cdot$};
	\node at (.1,0.2) {$\scriptstyle \cdot$};
	\draw [thick] (.2,.2)--(1.8,.2);
	\node at (1.9,.2) {$\scriptstyle \cdot$};
	\node at (2,.2) {$\scriptstyle \cdot$};
	\node at (2.1,.2) {$\scriptstyle \cdot$};
	\node at (1.5,.2) {$\bullet$};
	\node at (-.5,0) {$\scriptscriptstyle I$};
	\node at (-.1,0) {$\scriptstyle \cdot$};
	\node at (0,0) {$\scriptstyle \cdot$};
	\node at (.1,0) {$\scriptstyle \cdot$};
	\draw [thick] (.2,0)--(1.8,0);
	\node at (1.9,0) {$\scriptstyle \cdot$};
	\node at (2,0) {$\scriptstyle \cdot$};
	\node at (2.1,0) {$\scriptstyle \cdot$};
	\node at (.5,0) {$\bullet$};
	\node at (.5,-.2) {$\scriptscriptstyle O(I)$};
	\node at (1.5,.4) {$\scriptscriptstyle O(J)$}; 
	\end{tikzpicture}
\]
\end{proposition}

\begin{proof}
If the orientation~$O$ contains this pattern, it is clearly cyclic.
Conversely, assume that~$O$ is cyclic.
Then we can find $I_1, \dots, I_k \in \II$ with~$k \ge 2$ such that~$O(I_{i+1}) \in I_i \ssm \{O(I_i)\}$ for~$i \in [k-1]$ and~$O(I_1) \in I_k \ssm \{O(I_k)\}$.
Graphically
\[
	\begin{tikzpicture}[scale=1,baseline=.0cm]
	\node at (-.5,1) {$\scriptscriptstyle I_k$};
	\draw [thick,{Bar[width=3pt]}-] (.3,1)--(1.2,1);
	\node at (1.3,1) {$\scriptstyle \cdot$};
	\node at (1.4,1) {$\scriptstyle \cdot$};
	\node at (1.5,1) {$\scriptstyle \cdot$};
	\node at (1,0.9) {$\scriptstyle \cdot$};
	\node at (1.1,0.8) {$\scriptstyle \cdot$};
	\node at (1.2,0.7) {$\scriptstyle \cdot$}; 
	\node at (-.5,0.4) {$\scriptscriptstyle I_3$};
	\draw [thick,{Bar[width=3pt]}-{Bar[width=3pt]}] (2.5,.4)--(4.5,.4);
	\node at (3.5,.4) {$\bullet$};
	\node at (-.5,0.2) {$\scriptscriptstyle I_2$};
	\draw [thick,{Bar[width=3pt]}-{Bar[width=3pt]}] (1,.2)--(4,.2);
	\node at (1.5,.2) {$\bullet$};
	\node at (-.5,0) {$\scriptscriptstyle I_1$};
	\draw [thick,{Bar[width=3pt]}-{Bar[width=3pt]}] (0,0)--(1.5,0);
	\node at (.5,0) {$\bullet$};
	\end{tikzpicture}
\]
Assume that~$k > 2$ and is minimal for this property.
Note that~$O(I_i) \ne O(I_{i+1})$ for all~$i$.
By symmetry, suppose~$O(I_1) < O(I_2)$.
If there is~$i \in [k-1]$ such that~$O(I_{i+1}) < O(I_i)$, then for the smallest such~$i$, we have
\begin{itemize}
\item either~$O(I_{i-1}) \in [O(I_{i+1}), O(I_i)] \subseteq I_i$, so that~$O(I_{i-1}) \in I_i$ and $O(I_i) \in I_{i-1}$,
\item or~$O(I_{i+1}) \in [O(I_{i-1}), O(I_i)] \subseteq I_{i-1}$, so that we can drop~$I_i$ from our sequence, contradicting the minimality of~$k$.
\end{itemize}
We thus obtain~$O(I_1) < \dots < O(I_k)$.
As~$O(I_1) \in I_k$, we get~$O(I_{k-1}) \in [O(I_1), O(I_k)] \subseteq I_k$, so that~$O(I_{k-1}) \in I_k$ and~$O(I_k) \in I_{k-1}$.
\end{proof}

\begin{remark}
\cref{prop:acyclicI} fails when~$\HH$ is not an interval hypergraph.
For instance, the orientation~$O$ of the hypergraph~$\HH \eqdef \{ 1, 2, 3, 12, 23, 13 \}$ defined by~$O(12) = 1$, $O(23) = 2$ and~$O(13) = 3$ has a $3$-cycle but no $2$-cycle.
\end{remark}


\subsection{Fibers of $\Or$ for interval hypergraphs} 
\label{subsec:preimageI}

One striking property for interval hypergraphs is that the fibers of the surjection~$\Or$ are intervals in the weak order.
To describe this we first need to recall the following classical result of A.~Bj\"orner and M.~Wachs~\cite[Thm.~6.8]{BjornerWachs}

\begin{proposition}[{\cite[Thm.~6.8]{BjornerWachs}}]
\label{prop:WOIP}
The set of linear extensions of a poset~$\less$ on~$[n]$ forms an interval~$[\sigma, \tau]$ of the weak order if and only if ${a \less c \implies a \less b \text{ or } b \less c}$, and~${a \more c \implies a \more b \text{ or } b \more c}$, for every~$1 \le a < b < c \le n$.
Moreover, the inversions of~$\sigma$ are the pairs~$(b,a)$ with~$a < b$ and~$a \more b$, and the non-inversions of~$\tau$ are the pairs~$(a,b)$ with~$a < b$ and~$a \less b$.
\end{proposition}

This statement was refined in~\cite{ChatelPilaudPons} to describe Tamari interval posets.

\begin{proposition}[{\cite[Coro.~2.24]{ChatelPilaudPons}}]
\label{prop:TOIP}
The set of linear extensions of a poset~$\less$ on~$[n]$ forms an interval~$[\sigma, \tau]$ of the weak order such that~$\sigma$ avoids the pattern $231$ and $\tau$ avoids the pattern~$213$ if and only if~${a \less c \implies a \less b}$ and~$a \more c \implies b \more c$ for every~$1 \le a < b < c \le n$.
\end{proposition}

\begin{proposition}
\label{prop:preimageI}
The fiber~$\Or^{-1}(A)$ of any acyclic orientation~$A$ of an interval hypergraph~$\II$ is an interval of the weak order with minimum avoiding the pattern $231$ and maximum avoiding the pattern $213$ (in other words, a Tamari interval).
\end{proposition}

\begin{proof}
From \cref{lem:prepi}, the fiber~$\Or^{-1}(A)$ is the set of linear extensions of $\less_A$, so that we use the characterization of \cref{prop:TOIP} to prove \cref{prop:preimageI}.
Let $1 \le a < b < c \le [n]$.
If we have $a \less_{A} c$, then, by definition of $\less_A$, there must be a sequence $I_1, \dots, I_k \in \II$ such that~$a = A(I_1)$, $A(I_{i+1}) \in I_i$ for all~$i \in [k-1]$, and~$c \in I_k$.
As $\bigcup_{i \in [k]} I_i$ is an interval containing~$a$ and~$c$ and~$a < b < c$, it contains also~$b$.
Hence, there is~$i \in [k]$ such that~$b \in I_i$, and the sequence~$I_1, \dots, I_i$ proves that~$a \less_{A} b$.
The case $a \more_{A} c$ is similar and  implies that $b \more_A c$.
\end{proof}

\begin{example}
For the hypergraph~$\II = \{ 1, 2, 3, 4,123, 23, 234, 1234 \}$ of \cref{exm:intervalHypergraph}, the fibers of the surjection map~$\Or$ are represented in \cref{fig:exmInterval}\,(bottom left).
\end{example}

\begin{remark}
\label{prop:surjectionFactorizeI}
It follows from \cref{prop:preimageI} that the surjective poset morphism~$\Or$ factorizes through the Tamari lattice.
\end{remark}

\begin{remark}
\cref{prop:preimageI} fails when~$\HH$ is not an interval hypergraph.
For~$\HH \eqdef \{1, 2, 3, 13\}$, there are two fibers~$\{123, 213, 132\}$ and~$\{231, 312, 321\}$ which are not intervals of the weak order.
\end{remark}


\subsection{Source characterization for interval hypergraphs}  
\label{subsec:sourceAcyclicI}

We now characterize the comparisons in the poset $P_\II$ in terms of the comparisons of the sources for each $I\in \II$.

\begin{proposition}
\label{prop:sourceOrderI}
For any acyclic orientations~$A$ and~$B$ of an interval hypergraph~$\II$,
\[
A \le B \text{ in } P_\II \quad \iff \quad A(I) \le B(I) \text{ for all } I \in \II .
\]
\end{proposition}

\begin{proof}
The forward direction is immediate as it holds for increasing flip by \cref{def:flip} and any cover of~$P_\II$ is an increasing flip.
For the backward direction, assume that~$A(I) \le B(I)$ for all~$I \in \II$.
The proof works by induction on~$|\set{I \in \II}{A(I) < B(I)}|$.

Choose~$J \in \II$ such that~$A(J) < B(J)$ and for any~$I \in \II$ with~$A(I) < B(I)$, we have~$B(I) < B(J)$, or~$B(I) = B(J)$ and~$A(I) \le A(J)$.

Let~$O$ be the orientation of~$\II$ obtained from~$A$ by flipping~$A(J)$ to~$B(J)$ as described in \cref{def:flip}.
We claim that~$O$ is acyclic and that~$O(I) \le B(I)$ for all~$I \in \II$. 
We conclude by induction that~$O \le B$, and thus~$A \le O \le B$ as desired.

We first prove that~$O$ is acyclic.
Otherwise, we would have~$I,I' \in \II$ such that~$O(I) \in I' \ssm \{O(I')\}$ and~$O(I') \in I \ssm \{O(I)\}$.
As~$A$ is acyclic, we have~$A(I) \ne O(I)$ or~$A(I') \ne O(I')$, but not both since~$O(I) \ne O(I')$.
Hence, we can assume by symmetry that~$O(I) = A(I)$ while~$A(I') = A(J)$ and~$O(I') = B(J)$.
Up to updating~$J$ to~$I'$, we can thus also assume that~$I' = J$.
As~$B$ is acyclic, we have~$B(I) \ne O(I)$.
Since~$A(I) = O(I) \ne B(I)$, our choice of~$J$ ensures that
\begin{itemize}
\item either~$B(I) < B(J)$. 
We then obtain that~$A(I) \le B(I) < B(J)$.
As~$A(I) = O(I) \in I' = J$ and~$B(J) \in J$, we thus get that~$B(I) \in J$.
Moreover, $B(J) = O(I') \in I$.
As~$B(I) \ne B(J)$, we obtain a contradiction with the acyclicity of~$B$.
\item or~$B(I) = B(J)$ and~$A(I) \le A(J)$.
We then have~$A(I) \le A(J) \le B(J) = B(I)$ so that~$A(J) \in I$.
As~$A(I) = O(I) \in I' = J$ and~$A(I) \ne A(J)$, we obtain a contradiction with the acyclicity of~$A$.
\end{itemize}

We now prove that~$O(I) \le B(I)$ for all~$I \in \II$.
We thus consider~$I \in \II$ and distinguish two cases:
\begin{itemize}
\item Assume first that~$A(I) = A(J)$ and~$B(J) \in I$. By \cref{def:flip}, we then have~${O(I) = B(J)}$. Moreover, as~$I$ is an interval and contains~$A(J)$ and~$B(J)$, it contains~$[A(J), B(J)] \subseteq J$. As~$B$ is acyclic, this implies that~$B(I) \notin {[A(J), B(J)[}$. As~$A(J) = A(I) \le B(I)$, we thus obtain that~$O(I) = B(J) \le B(I)$.
\item Otherwise, we have~$O(I) = A(I) \le B(I)$.
\qedhere
\end{itemize}
\end{proof}

\begin{remark}
After the completion of this paper, \cref{prop:sourceOrderI} was extended to arbitrary hypergraphs by F.~Gélinas in~\cite{Gelinas}.
\end{remark}

\begin{remark}
Note that \cref{prop:sourceOrderI} implies that the hypergraphic poset~$P_\II$ has Dushnik-Miller dimension at most~$|\II|$.
\end{remark}


\subsection{Flips and cover relations for interval hypergraphs}  
\label{subsec:cover}

In this section, we exploit \cref{prop:acyclicI} to provide a simple description of the flips and cover relations of~$P_\II$ for an interval hypergraph~$\II$.

\begin{proposition}
\label{prop:isFlipI}
Consider an acyclic orientation~$A$ of an interval hypergraph~$\II$, an increasing flip~$\flip{A}{i}{j}{B}$ (in the sense of \cref{def:flip}), and let~$k \eqdef \max\set{\max(I)}{I \in \II \text{ and } A(I) = i}$.
Then~$B$ is acyclic if and only if there is no~$J \in \II$ with~$j \in J \ssm \{A(J)\}$ and~$A(J) \in {]i,k]}$.
A symmetric statement holds for decreasing flips.
\end{proposition}

\begin{proof}
We prove the result for increasing flips, the result for decreasing flips follows by the symmetry of~\cref{prop:antiIsomorphism}.

\pagebreak
Assume first that there is~$J \in \II$ with~$j \in J \ssm \{A(J)\}$ and~$A(J) \in {]i,k]}$.
Let~$I \in \II$ be such that~$A(I) = i$ and~$k = \max(I)$.
Then~$B(J) = A(J) \in {]i,k]} = {]A(I), \max(I)]} \subseteq I$ and~$B(I) = j \in J$ and~$B(I) = j \ne A(J) = B(J)$ implies that $B$ is cyclic.

Conversely, assume that $B$ is cyclic.
By \cref{prop:acyclicI}, there are~$I,J \in \II$ with~${B(I) \in J \ssm \{B(J)\}}$ and~$B(J) \in I \ssm \{B(I)\}$.
As~$A$ is acyclic, we have~$i = A(I) \ne B(I) = j$ or~$i = A(J) \ne B(J) = j$, but not both since~$B(I) \ne B(J)$.
By symmetry, we can assume that~$i = A(I) \ne B(I) = j$ and~$A(J) = B(J)$.
We obtain that~$j = B(I) \in J \ssm \{B(J)\} = J \ssm \{A(J)\}$.
Moreover, we have~$i < A(J)$ as otherwise~$A(J) = B(J) \in I$ and~$A(I) = i \in [A(J), j] = [A(J), B(I)] \subseteq J$ and~$A(I) \ne A(J)$ would contradict the acyclicity of~$A$.
As~$A(J) = B(J) \in I \ssm \{B(I)\}$, we obtain that~$i < A(J) < \max(I)$.
We conclude that~$A(J) \in {]i,k]}$ since~$\max(I) \le k$ as~$A(I) = i$.
\end{proof}

\begin{proposition}
\label{prop:alwaysFlippableI}
Consider an acyclic orientation~$A$ of an interval hypergraph~$\II$ and~$i \in [n]$ such that there is~$I \in \II$ with~$i = A(I) < \max(I)$.
Then there exists~$j > i$ such that the orientation of~$\II$ obtained by flipping $i$ to~$j$ is acyclic.
A symmetric statement holds with a decreasing flip if~$i = A(I) > \min(I)$.
\end{proposition}

\begin{proof}
We prove the result for increasing flips, the result for decreasing flips follows by the symmetry of~\cref{prop:antiIsomorphism}.

Let~$k \eqdef \max\set{\max(I)}{I \in \II \text{ and } A(I) = i}$ and consider
\[
X \eqdef \bigcup_{\substack{J \in \II \\ A(J) \in {]i,k]}}} J \ssm \{A(J)\}.
\]
Assume that~$X = {]i,k]}$, and consider a minimal~$j \in {]i,k]}$ for the poset~$\less_A$.
As~$j \in X$, there is~$J \in \II$ such that~$A(J) \in {]i,k]}$ and~$j \in J \ssm \{A(J)\}$.
Hence, we have~$A(J) \less_A j$ and~$A(J) \in {]i,k]}$, a contradiction.
We conclude that there is~$j \in {]i,k]} \ssm X$.
By \cref{prop:isFlipI}, the orientation obtained from~$A$ by flipping of~$i$ to~$j$ is acyclic.
\end{proof}

\begin{remark}
\cref{prop:alwaysFlippableI} fails when~$\HH$ is not an interval hypergraph.
For instance, for the hypergraph $\HH=\{ 1, 2, 3, 4, 123, 134 \}$ of \cref{exm:DH1,exm:DH2,exm:DH3,exm:DH4,exm:DH5}, and the acyclic orientation~$A$ defined by~$A(123) = 2$ and~$A(134) = 1$, we have~$2 = A(123) < \max(123)$, but no increasing flip from~$A$ to an acyclic orientation of~$\HH$ flips~$2$.
\end{remark}

\begin{proposition}
\label{prop:isCoverI}
An increasing flip \flip{A}{i}{j}{B} between two acyclic orientations of an interval hypergraph~$\II$ is a cover relation of~$P_\II$ if and only if
\[
{]i,j[} \;\; \subseteq \bigcup_{\substack{J \in \II \\ A(J) \in {]i,j]}}} J \ssm \{A(J)\}.
\]
\end{proposition}

\begin{proof}
Let~$k \eqdef \max\set{\max(I)}{I \in \II \text{ and } A(I) = i}$ and~$I \in \II$ be such that~$A(I) = i$ and~${k = \max(I)}$.
Note that~$i < j \le k$ since~$\flip{A}{i}{j}{B}$ is a flip.

Assume first that there is~$\ell \in {]i,j[} \ssm \bigcup_{A(J) \in {]i,j]}} J \ssm \{A(J)\}$.
We claim that there is no~$J \in \II$ with~$\ell \in J \ssm A(J)$ and~$A(J) \in {]i,k]}$.
Indeed, by definition of~$\ell$, we would have~$A(J) > j$.
Then~$B(I) = j \in [\ell, A(J)] \subseteq J$ and~$B(J) = A(J) \in {]i,k]} \subseteq I$ and~$B(I) = j < A(J) = B(J)$ contradicts the acyclicity of~$A$.
We conclude from \cref{prop:isFlipI} that the orientation~$C$ of~$\II$ obtained by flipping~$i$ to~$\ell$ is acyclic.
This implies that~$\flip{A}{i}{j}{B}$ is not a cover relation as it is obtained transitively from~$\flip{A}{i}{\ell}{C}$ and~$\flip{C}{\ell}{j}{B}$.

Conversely, assume that~$A \le B$ is not a cover relation.
Then there is a flip~$\flip{A}{i}{\ell}{C}$, where~$C$ is an acyclic orientation of~$\II$ such that~$A < C < B$.
Since~$\flip{A}{i}{j}{B}$ is a flip, we have~${A(K) = C(K) = B(K)}$ for any~$K \in \II$, except if~$A(K) = i$ and~$B(K) = j$.
Hence, there is~$K \in \II$ such that~$A(I) = i < C(I) = \ell < j = B(I)$.
Since~$\flip{A}{i}{\ell}{C}$ is a flip, we obtain from \cref{prop:isFlipI} that there is no~$J \in \II$ with~$\ell \in J \ssm \{A(J)\}$ and~$A(J) \in {]i,k]} \subseteq {]i,j]}$.
Hence,~$\ell \in {]i,j[} \ssm \bigcup_{A(J) \in {]i,j]}} J \ssm \{A(J)\}$.
\end{proof}


\pagebreak
\section{Interval hypergraphic lattices}
\label{sec:LatticeI}

\begin{figure}[b]
	\centerline{
		\begin{tabular}{c@{\qquad}c}
			\begin{tikzpicture}[scale=1.5,baseline=-3.5cm]
				\node (a) at (2,0) {\acyclicOrientation{4}{1/3/1,2/4/2}};
				\node (b) at (1,1) {\acyclicOrientation{4}{1/3/1,2/4/3}};
				\node (c) at (3,1) {\acyclicOrientation{4}{1/3/2,2/4/2}};
				\node (d) at (1,2) {\acyclicOrientation{4}{1/3/1,2/4/4}};
				\node (e) at (3,3) {\acyclicOrientation{4}{1/3/3,2/4/3}};
				\node (f) at (1,3) {\acyclicOrientation{4}{1/3/2,2/4/4}};
				\node (g) at (2,4) {\acyclicOrientation{4}{1/3/3,2/4/4}};
				\draw (a)--(b);
				\draw (a)--(c);
				\draw (b)--(d);
				\draw (b)--(e);
				\draw (c)--(e);
				\draw (c)--(f);
				\draw (d)--(f);
				\draw (f)--(g);
				\draw (e)--(g);
			\end{tikzpicture}
			&
			\begin{tikzpicture}[scale=2.5]
				\node (a) at (2.5,0) {\acyclicOrientation{4}{1/2/1,1/3/1,1/4/1,2/4/2,3/4/3}};
				\node (b) at (1,1) {\acyclicOrientation{4}{1/2/1,1/3/1,1/4/1,2/4/3,3/4/3}};
				\node (c) at (2.5,1) {\acyclicOrientation{4}{1/2/1,1/3/1,1/4/1,2/4/2,3/4/4}};
				\node (d) at (4,1) {\acyclicOrientation{4}{1/2/2,1/3/2,1/4/2,2/4/2,3/4/3}};
				\node (e) at (1,2) {\acyclicOrientation{4}{1/2/1,1/3/1,1/4/1,2/4/4,3/4/4}};
				\node (f) at (2,2) {\acyclicOrientation{4}{1/2/1,1/3/3,1/4/3,2/4/3,3/4/3}};
				\node (g) at (4,2) {\acyclicOrientation{4}{1/2/2,1/3/2,1/4/2,2/4/2,3/4/4}};
				\node (h) at (1,3) {\acyclicOrientation{4}{1/2/1,1/3/1,1/4/4,2/4/4,3/4/4}};
				\node (i) at (3,3) {\acyclicOrientation{4}{1/2/2,1/3/3,1/4/3,2/4/3,3/4/3}};
				\node (j) at (2,4) {\acyclicOrientation{4}{1/2/1,1/3/2,1/4/4,2/4/4,3/4/4}};
				\node (k) at (4,4) {\acyclicOrientation{4}{1/2/2,1/3/2,1/4/4,2/4/4,3/4/4}};
				\node (l) at (3,5) {\acyclicOrientation{4}{1/2/2,1/3/3,1/4/4,2/4/4,3/4/4}};
				\draw (a)--(b);
				\draw (a)--(c);
				\draw (a)--(d);
				\draw (b)--(e);
				\draw (b)--(f);
				\draw (c)--(e);
				\draw (c)--(g);
				\draw (d)--(i);
				\draw (d)--(g);
				\draw (e)--(h);
				\draw (f)--(i);
				\draw (f)--(j);
				\draw (g)--(k);
				\draw (h)--(j);
				\draw (h)--(k);
				\draw (i)--(l);
				\draw (j)--(l);
				\draw (k)--(l);
			\end{tikzpicture}
			\\[.2cm]
			$\{1,2,3,4,123,234\}$
			&
			$\{1,2,3,4,12,123,1234,234,34\}$
		\end{tabular}
	}
	\caption{Two interval hypergraphic posets which are not lattices.}
	\label{fig:notLattices}
\end{figure}

\begin{figure}
	\centerline{
		\begin{tabular}{c@{\qquad}c}
			\begin{tikzpicture}[scale=1.5,baseline=-3.5cm]
				\node (a) at (1,0) {\acyclicOrientation{3}{1/2/1,2/3/2,1/3/1}};
				\node (b) at (0,1.3) {\acyclicOrientation{3}{1/2/1,2/3/3,1/3/1}};
				\node (c) at (2,2) {\acyclicOrientation{3}{1/2/2,2/3/2,1/3/2}};
				\node (d) at (0,2.7) {\acyclicOrientation{3}{1/2/1,2/3/3,1/3/3}};
				\node (e) at (1,4) {\acyclicOrientation{3}{1/2/2,2/3/3,1/3/3}};
				\draw (a)--(b);
				\draw (a)--(c);
				\draw (b)--(d);
				\draw (c)--(e);
				\draw (d)--(e);
			\end{tikzpicture}
			&
			\begin{tikzpicture}[scale=2.2]
				\node (a) at (3,0) {\acyclicOrientation{4}{1/2/1,2/3/2,3/4/3,1/3/1,2/4/2,1/4/1}};
				\node (b) at (1,1) {\acyclicOrientation{4}{1/2/2,2/3/2,3/4/3,1/3/2,2/4/2,1/4/2}};
				\node (c) at (3,1) {\acyclicOrientation{4}{1/2/1,2/3/3,3/4/3,1/3/1,2/4/3,1/4/1}};
				\node (d) at (5,1) {\acyclicOrientation{4}{1/2/1,2/3/2,3/4/4,1/3/1,2/4/2,1/4/1}};
				\node (e) at (3,3) {\acyclicOrientation{4}{1/2/2,2/3/2,3/4/4,1/3/2,2/4/2,1/4/2}};
				\node (f) at (5,2) {\acyclicOrientation{4}{1/2/1,2/3/2,3/4/4,1/3/1,2/4/4,1/4/1}};
				\node (g) at (2,4) {\acyclicOrientation{4}{1/2/1,2/3/3,3/4/3,1/3/3,2/4/3,1/4/3}};
				\node (h) at (4,3) {\acyclicOrientation{4}{1/2/1,2/3/2,3/4/4,1/3/1,2/4/4,1/4/4}};
				\node (i) at (5,3) {\acyclicOrientation{4}{1/2/1,2/3/3,3/4/4,1/3/1,2/4/4,1/4/1}};
				\node (j) at (1,5) {\acyclicOrientation{4}{1/2/2,2/3/3,3/4/3,1/3/3,2/4/3,1/4/3}};
				\node (k) at (3,5) {\acyclicOrientation{4}{1/2/2,2/3/2,3/4/4,1/3/2,2/4/4,1/4/4}};
				\node (l) at (5,4) {\acyclicOrientation{4}{1/2/1,2/3/3,3/4/4,1/3/1,2/4/4,1/4/4}};
				\node (m) at (5,5) {\acyclicOrientation{4}{1/2/1,2/3/3,3/4/4,1/3/3,2/4/4,1/4/4}};
				\node (n) at (3,6) {\acyclicOrientation{4}{1/2/2,2/3/3,3/4/4,1/3/3,2/4/4,1/4/4}};
				\draw (a)--(b);
				\draw (a)--(c);
				\draw (a)--(d);
				\draw (b)--(e);
				\draw (b)--(j);
				\draw (c)--(g);
				\draw (c)--(i);
				\draw (d)--(e);
				\draw (d)--(f);
				\draw (e)--(k);
				\draw (f)--(h);
				\draw (f)--(i);
				\draw (g)--(j);
				\draw (g)--(m);
				\draw (h)--(k);
				\draw (h)--(l);
				\draw (i)--(l);
				\draw (j)--(n);
				\draw (k)--(n);
				\draw (l)--(m);
				\draw (m)--(n);
			\end{tikzpicture}
			\\[.2cm]
			$\{1,2,3,12,23,123\}$
			&
			$\{1,2,3,4,12,23,34,123,234,1234\}$
		\end{tabular}
	}
	\caption{The Tamari lattice (semidistributive lattice, but not distributive).}
	\label{fig:Tamari}
\end{figure}

In this section, we prove \cref{thm:latticeI} which we first introduce properly:

\begin{definition}
\label{def:intersectionClosed}
An interval hypergraph~$\II$ is \defn{closed under intersection} if $I, J \in \II$ and~$I \cap J \ne \varnothing$ implies~$I \cap J \in \II$.
\end{definition}

\begin{definition}
A poset~$P$ is a \defn{lattice} if any subset of~$P$ admits a \defn{join} (least upper bound) and a \defn{meet} (greatest lower bound).
\end{definition}

\begin{theoremA}
For an interval hypergraph $\II$ on~$[n]$ (with our convention that~$\{i\} \in \II$ for all~$i \in [n]$), the poset $P_\II$ is a lattice if and only if $\II$ is closed under intersection.
\end{theoremA}

\begin{example}
The interval hypergraphic posets of \cref{fig:notLattices} are not lattices, while those of \cref{fig:Tamari,fig:distributiveLattices,fig:semidistributiveLattices,fig:notSemidistributiveLattices} are lattices.
\end{example}

\begin{example}
The hypergraphic poset of~$\{ 1, 2, 3, 12, 23, 123\}$ (resp.~of~$\{ 1, 2, 3, 12, 123\}$) is a lattice.
In general, the hypergraphic poset of all intervals (resp. all initial intervals) is the Tamari lattice (resp.~the boolean lattice).
See \cref{fig:Tamari,fig:distributiveLattices}\,(bottom right).
\end{example}

\begin{example}
The hypergraphic posets of~$\{1, 2, 3, 4, 123, 234\}$ and of~$\{ 1, 2, 3, 4, 12, 123, 1234, 234, 34\}$ are not lattices.
In general, the hypergraphic poset of all initial and final intervals is not a lattice.
See \cref{fig:notLattices}.
\end{example}


\subsection{If $P_\II$ is a lattice, then $\II$ is closed under intersection}  
\label{subsec:latticeForward}

We are now ready to show the forward implication of \cref{thm:latticeI}.
See \cref{fig:notLattices} for an illustration.

\begin{proposition}
\label{prop:latticeForward}
If~$\II$ is an interval  hypergraph such that the poset $P_\II$ is a lattice, then $\II$ is closed under intersection.
\end{proposition}

\begin{proof}
See \cref{exm:proofLatticeForward} for the smallest example of the proof.
By contradiction, assume we have $I,J\in \II$ such that $\varnothing  \not = I\cap J\not\in \II$.
Define~$1 \le a < b < c < d \le n$ by
\begin{alignat*}{6}
& a \eqdef b-1, &
\qquad
& b \eqdef \min(I\cap J), &
\qquad
& c \eqdef \max(I\cap J), &
\qquad
& d \eqdef c+1. &
\intertext{
Note that $b\ne c$ since~$\II$ contains all singletons, hence~$1 \le a < b < c < d \le n$.
By symmetry, we assume that $a\in I\ssm J$ and $d\in J\ssm I$.
Let $X$ be the word formed by the complement of $\{a,b,c,d\}$ in~$[n]$ written in increasing order.
We now construct four permutations 
}
& \pi_A \eqdef bacdX, &
\qquad
& \pi_B \eqdef acdbX,&
\qquad
& \pi_C \eqdef dbacX, &
\qquad
& \pi_D \eqdef cdbaX, &
\intertext{
and consider the four distinct acyclic orientations
}
& A \eqdef \Or_{\pi_A}, &
\qquad
& B \eqdef \Or_{\pi_B}, &
\qquad
& C \eqdef \Or_{\pi_C}, &
\qquad
& D \eqdef \Or_{\pi_D}. &
\intertext{
We display below the four orientations highlighting only the intervals $I$ and $J$:
}
&
	A =  
	\begin{tikzpicture}[scale=1,baseline=.0cm]
	\draw [thick,{Bar[width=3pt]}-{Bar[width=3pt]}] (0,0)--(1.2,0);   \node at (.3,0) {$\bullet$};
	\draw [thick,{Bar[width=3pt]}-{Bar[width=3pt]}] (.3,.3)--(1.5,.3);   \node at (.3,.3) {$\bullet$};
	\node at (1.8,0) {$\scriptstyle I$};
	\node at (1.8,.3) {$\scriptstyle J$};
	\node at  (0,-.33) {$\scriptstyle a$};
	\node at  (.3,-.3) {$\scriptstyle b$};
	\node at  (1.2,-.33) {$\scriptstyle c$};
	\node at  (1.5,-.3) {$\scriptstyle d$};
	\end{tikzpicture} 
&&
	B =
	\begin{tikzpicture}[scale=1,baseline=.0cm]
	\draw [thick,{Bar[width=3pt]}-{Bar[width=3pt]}] (0,0)--(1.2,0);   \node at (0,0) {$\bullet$};
	\draw [thick,{Bar[width=3pt]}-{Bar[width=3pt]}] (.3,.3)--(1.5,.3);   \node at (1.2,.3) {$\bullet$};
	\node at (1.8,0) {$\scriptstyle I$};
	\node at (1.8,.3) {$\scriptstyle J$};
	\node at  (0,-.33) {$\scriptstyle a$};
	\node at  (.3,-.3) {$\scriptstyle b$};
	\node at  (1.2,-.33) {$\scriptstyle c$};
	\node at  (1.5,-.3) {$\scriptstyle d$};
	\end{tikzpicture} 
&&
	C =  
	\begin{tikzpicture}[scale=1,baseline=.0cm]
	\draw [thick,{Bar[width=3pt]}-{Bar[width=3pt]}] (0,0)--(1.2,0);   \node at (.3,0) {$\bullet$};
	\draw [thick,{Bar[width=3pt]}-{Bar[width=3pt]}] (.3,.3)--(1.5,.3);   \node at (1.5,.3) {$\bullet$};
	\node at (1.8,0) {$\scriptstyle I$};
	\node at (1.8,.3) {$\scriptstyle J$};
	\node at  (0,-.33) {$\scriptstyle a$};
	\node at  (.3,-.3) {$\scriptstyle b$};
	\node at  (1.2,-.33) {$\scriptstyle c$};
	\node at  (1.5,-.3) {$\scriptstyle d$};
	\end{tikzpicture} 
&&
	D =
	\begin{tikzpicture}[scale=1,baseline=.0cm]
	\draw [thick,{Bar[width=3pt]}-{Bar[width=3pt]}] (0,0)--(1.2,0);   \node at (1.2,0) {$\bullet$};
	\draw [thick,{Bar[width=3pt]}-{Bar[width=3pt]}] (.3,.3)--(1.5,.3);   \node at (1.2,.3) {$\bullet$};
	\node at (1.8,0) {$\scriptstyle I$};
	\node at (1.8,.3) {$\scriptstyle J$};
	\node at  (0,-.33) {$\scriptstyle a$};
	\node at  (.3,-.3) {$\scriptstyle b$};
	\node at  (1.2,-.33) {$\scriptstyle c$};
	\node at  (1.5,-.3) {$\scriptstyle d$};
	\end{tikzpicture} 
&
\end{alignat*}
We have that $\pi_A < \pi_C$, $\pi_A < \pi_D$ and $\pi_B < \pi_D$ in the weak order.
\cref{coro:weakToP} implies that $A < C$, $A < D$ and $B < D$ in~$P_\II$.
We moreover claim that $B < C$  but this does not follow directly from the weak order since $\pi_B\not<\pi_C$.
To show our claim, consider $\pi_E=adcbX$ and $\pi_F=adbcX$.
For any $K\in \II$ such that $a,d\not\in K$, we cannot have both $b,c\in K$, since this would imply that $K=[b,c]=I\cap J\not\in \II$, a contradiction.
This shows that $E=\Or_{\pi_E}=\Or_{\pi_F}=F$.
Now we have $\pi_B<\pi_E$ and $\pi_F<\pi_C$ which gives $B<E=F<C$.
 
 If the poset $P_\II$ is a lattice and~$A \le C$ $A \le D$, $B \le C$ and~$B \le D$, there is~$M$ such that~$A \le M$, $B \le M$, $M \le C$ and~$M \le D$ (anything between~$A \join B$ and~$C \meet D$ works).
Let $\pi_M$ be any permutation such that $M=\Or_{\pi_M}$ and let  
\[
m=\pi \big( \min\set{i}{\pi_M(i)\in I\cup J} \big).
\]
If $m<b$, then $M(I) = m < b = A(I)$ and the easy forward implication of \cref{prop:sourceOrderI} implies that $A\not\le M$.
By similar arguments $b\le m<c$ would imply~$B\not\le M$, $b< m\le c$ would imply~$M\not\le C$, and $c<m$ would imply~$M\not\le D$.
This is a contradiction to the existence of $m$.
\end{proof}

\begin{example}
\label{exm:proofLatticeForward}
For the hypergraph~$\{1, 2, 3, 4, 123, 234\}$ illustrated in \cref{fig:notLattices}\,(left), we have~$a = 1$, $b = 2$, $c = 3$, $d = 4$ and
\[
A = \acyclicOrientation{4}{1/3/2,2/4/2}
\qquad
B = \acyclicOrientation{4}{1/3/1,2/4/3}
\qquad
C = \acyclicOrientation{4}{1/3/2,2/4/4}
\quad\text{and}\quad
D = \acyclicOrientation{4}{1/3/3,2/4/3}.
\]
We have~$A, B \le C, D$ and there is no~$M$ with~$A,B \le M \le C,D$, so that~$A$ and~$B$ have no join, and~$C$ and~$D$ has no meet.
\end{example}

\begin{remark}
\cref{prop:latticeForward} fails when~$\HH$ is not an interval hypergraph.
For instance, the hypergraph~$\HH \eqdef \{ 1, 2, 3, 4, 5, 1234, 2345, 23, 24, 34 \}$ is not closed under intersection, while its hypergraphic poset~$P_\HH$ is a lattice.
\end{remark}


\subsection{Properties of $P_\II$ when~$\II$ is closed under intersection}  
\label{subsec:IntClosedI}

For the backward implication of \cref{thm:latticeI}, we need to investigate the properties of interval hypergraphs that are closed under intersection.
Recall from \cref{prop:preimageI} that the fiber $\Or^{-1}(A)$ of any acyclic orientation~$A$ of~$\II$ is an interval~$[\projDown_A,\projUp_A]$ in the weak order.
In the following we will keep this convention that $\projDown_A$ and $\projUp_A$ respectively denote the minimum and the maximum of the interval~$\Or^{-1}(A)$.

\begin{theorem}
\label{thm:propertiesI}
If $\II$ is an interval  hypergraph closed under intersection, then the following are equivalent for two acyclic orientations $A$ and $B$ of $\II$:
\begin{enumerate}[(a)]
	\item $A\le B$ in $P_\II$,
	\item $A(I) \le B(I)$ for all $I\in\II$,
	\item $\projDown_A \le \projUp_B$ in the weak order,
	\item $i \more_A j$ implies~$i \not\less_B j$ for all $i<j$.
\end{enumerate}
\end{theorem}

\pagebreak
\begin{proof}
The equivalence \mbox{(a)$\iff$(b)} was established in \cref{prop:sourceOrderI}.

For the equivalence \mbox{(c)$\iff$(d)}, \cref{prop:WOIP,prop:preimageI} show that the inversion set of $\projDown_A$ is $\set{(j,i)}{i<j, \, i \more_A j}$ while the inversion set of $\projUp_B$ is~$\set{(j,i)}{i<j, \, i \not\less_B j}$.
The equivalence of \mbox{(c)$\iff$(d)} thus follows from the characterization of the weak order in terms of inclusion of inversion sets.

The implication \mbox{(c)$\implies$(a)} follows from  \cref{coro:weakToP}.
Thus we only need to prove \mbox{(b)$\implies$(d)}.

For a contradiction, assume that $A(I) \le B(I)$ for all~$I \in \II$ and that we have some $i<j$ such that $i \more_A j$ and~$i \less_B j$.
Choose one such pair with~$j-i$ minimal.
Since $i \more_A j$ there are~$I_1, \dots, I_a \in \II$ such that~$j = A(I_1)$, $A(I_{p+1}) \in I_p$ for all~$p \in [a-1]$, and~$i \in I_a$.
Since $i \less_B j$ there are~$J_1, \dots, J_b \in \II$ such that~$i = B(J_1)$, $B(J_{q+1}) \in J_q$ for all~$q \in [b-1]$, and~$j \in J_b$.
Note that~$i \more_A A(I_p)$ for all~$p \in [a]$ and that~$B(J_q) \less_B j$ for all~$q \in [b]$.
Moreover, as~$\bigcup_{p \in [a]} I_p$ and~$\bigcup_{q \in [b]} J_q$ are both intervals containing~$i$ and~$j$, we have~$i \less_B k \more_A j$ for all~$i \le k \le j$.
By minimality of~$j-i$, we thus obtain that~$A(I_p) \notin {]i,j[}$ for~$p \in [a]$ and~$B(J_q) \notin {]i,j[}$ for~$q \in [b]$.
Hence, $[i,j]$ is contained in some~$I_p$ and some~$I_q$.
As~$A$ and~$B$ are acyclic, we obtain that~$a = 1 = b$.

We can thus assume that we have~$I,J \in \II$ such that~$i = B(J) \in I$ and~$j = A(I) \in J$.
As~$\II$ is closed under intersection, we have~$\{i,j\} \subseteq K \eqdef I \cap J \in \II$.
We have~$A(K) = j$, as otherwise~${A(I) = j \in K}$ an~$A(K) \in K \subseteq I$ and~$A(I) \ne A(K)$ would contradict the acyclicity of~$A$.
Similarly, we have~$B(K) = i$.
We conclude that~$B(K) = i < j = A(K)$ and~$K \in \II$ contradicting~(b).
\end{proof}

\begin{remark}
The implication \mbox{(a)$\implies$(c)} in \cref{thm:propertiesI} fails when~$\II$ is not closed under intersection.
For instance, for~$\II = \{1, 2, 3, 4, 123, 234\}$ and the acyclic orientations
\[
\begin{array}{ccccc}
	\acyclicOrientation{4}{1/3/1,2/4/3}
	&\qquad \le \qquad\qquad&
	\acyclicOrientation{4}{1/3/1,2/4/4}
	&\qquad \le \qquad\qquad&
	\acyclicOrientation{4}{1/3/2,2/4/3}
	\\[.6cm]
	A && B && C
\end{array}
\]
we have $[\projDown_A,\projUp_A]=[1324,1342]$, $[\projDown_B,\projUp_B]=[1423,4132]$ and $[\projDown_C,\projUp_C]=[4213,4231]$, so~$\projDown_A\not\le \projUp_C$.
\end{remark} 


\subsection{If $\II$ is closed under intersection, then $P_\II$ is a lattice}  
\label{subsec:latticeBackward}

We now conclude the proof of  \cref{thm:latticeI}.
This is a corollary of \cref{thm:propertiesI}.

\begin{proposition}
\label{prop:latticeBackward}
If~$\II$ is an interval  hypergraph closed under intersection, then the poset $P_\II$ is a lattice where
\[
A \join B =\Or_{\projDown_A \join \projDown_B}
\qquad\text{and}\qquad
A \meet B =\Or_{\projUp_A \meet \projUp_B}
\,.
\]
\end{proposition}

\begin{proof}
Consider four acyclic orientations~$A,B,C,D$ of~$\II$ such that $A\le C$, $A\le D$, $B\le C$ and~$B\le D$.
\cref{thm:propertiesI} implies that $\projDown_A \le \projUp_C$, $\projDown_A \le \projUp_D$, $\projDown_B \le \projUp_C$ and $\projDown_B \le \projUp_D$.
Hence
\begin{alignat*}{5}
& \projDown_A \le \projDown_A \join \projDown_B, &
\qquad
& \projDown_B \le \projDown_A \join \projDown_B, &
\qquad
& \projDown_A \join \projDown_B \le \projUp_C, &
\qquad
& \projDown_A \join \projDown_B \le \projUp_D, &
\intertext{which implies by \cref{coro:weakToP} that}
& A \le \Or_{\projDown_A \join \projDown_B}, &
\qquad
& B \le \Or_{\projDown_A \join \projDown_B}, &
\qquad
& \Or_{\projDown_A \join \projDown_B} \le C, &
\qquad
& \Or_{\projDown_A \join \projDown_B} \le D. &
\end{alignat*}
This implies that $A$ and $B$ admit a join
\[
A \join B \le \Or_{\projDown_A \join \projDown_B}.
\]
As~$A \le A\join B$ and $B \le A\join B$, \cref{thm:propertiesI} ensures that $\projDown_A \le \projUp_{A\join B}$ and~$\projDown_B \le \projUp_{A\join B}$, so that~$\projDown_A \join \projDown_B \le \projUp_{A\join B}$. By \cref{coro:weakToP}, this implies
\[
\Or_{\projDown_A \join \projDown_B} \le A \join B.
\]
This shows the result for the join~$\join$.
The proof for the meet~$\meet$ is similar.
\end{proof}

\pagebreak
\begin{example}
For the hypergraph~$\{ 1, 2, 3, 4,123, 23, 234, 1234 \}$ of \cref{exm:intervalHypergraph,fig:exmInterval}, we have
\[
\raisebox{-.2cm}{\acyclicOrientation{4}{1/3/2,1/4/2,2/3/2,2/4/2}} \join \raisebox{-.2cm}{\acyclicOrientation{4}{1/3/1,1/4/1,2/3/3,2/4/3}} = \Or_{2134 \join 1324} = \Or_{3214} = \raisebox{-.2cm}{\acyclicOrientation{4}{1/3/3,1/4/3,2/3/3,2/4/3}}
\]
and
\[
\raisebox{-.2cm}{\acyclicOrientation{4}{1/3/2,1/4/2,2/3/2,2/4/2}} \meet \raisebox{-.2cm}{\acyclicOrientation{4}{1/3/1,1/4/1,2/3/3,2/4/3}} = \Or_{2431 \meet 1342} = \Or_{1234} = \raisebox{-.2cm}{\acyclicOrientation{4}{1/3/1,1/4/1,2/3/2,2/4/2}}
\]
\end{example}

\begin{remark}
\cref{prop:latticeBackward} fails when~$\HH$ is not an interval hypergraph.
For instance, the hypergraph~$\HH \eqdef \{ 1, 2, 3, 4, 12, 13, 24, 34 \}$ is closed under intersection, while its hypergraphic poset~$P_\HH$ is not a lattice.
More generally, all hypergraphs~$\HH$ with~$|H| \le 2$ for all~$H \in \HH$ are closed under intersection, and the graphical zonotopes whose oriented skeleta are lattices were characterized in~\cite{Pilaud-acyclicReorientationLattices}.
\end{remark}

\begin{remark}
In fact, our proof of \cref{prop:latticeBackward} is a general statement about \defn{quasi lattice maps}.
Namely, if~$L$ is a lattice and~$P$ a poset, and $\Psi : L \to P$ is a poset morphism such that~$\Psi^{-1}(A) = [\projDown_A, \projUp_A]$ is an interval  for all $A\in P$, and $A \le B$ implies $\projDown_A \le \projUp_B$, then $P$ is a lattice where~$A \join B =\Psi(\projDown_A \join \projDown_B)$ and~$A \meet B =\Psi(\projUp_A \meet \projUp_B)$.
Note that lattice maps satisfy the stronger condition that~$A \le B$ implies~$\projDown_A \le \projDown_B$ and~$\projUp_A \le \projUp_B$.
It would be interesting to characterize the quasi lattice maps of the weak order.
Note that the lattice maps of the weak order where described by N.~Reading in~\cite{Reading-latticeCongruences, Reading-arcDiagrams}.
\end{remark}

\begin{proposition}
\label{prop:joinLattice}
For an interval hypergraph~$\II$ closed under intersection, any acyclic orientations~$A_1, \dots A_q$ of~$\II$, and any~${I \in \II}$, we have
\[
\Big( \bigJoin_{p \in [q]} A_p \Big)(I) = \min \bigg( I \ssm \Big( \bigcup_{p \in [q]} \bigcup_{\substack{J \in \II \\ A_p(J) \in I}} {[\min(J), A_p(J)[} \Big) \bigg).
\]
\end{proposition}

\begin{proof}
Let~$O$ be the orientation of~$\II$ defined by the formula of the statement.
We first prove that~$O$ is acyclic.
Otherwise, \cref{prop:acyclicI} ensures the existence of~$I,J \in \II$ such that~$O(I) \in J \ssm \{O(J)\}$ and~$O(J) \in I \ssm \{O(I)\}$.
Assume by symmetry that~$O(I) < O(J)$.
As~$O(I) \in J$ and~$O(J) > O(I)$, there exists~$p \in [q]$ and~$K \in \II$ such that~$A_p(K) \in J$ and~$\min(K) \le O(I) < A_p(K) \le O(J)$.
As~$I$ is an interval containing~$O(I)$ and~$O(J)$, and~$O(I) < A_p(K) \le O(J)$, it also contains~$A_p(K)$.
We thus obtain that~$A_p(K) \in I$ and~$\min(K) \le O(I) < A_p(K)$ contradicting our definition of~$O$.

We now prove that~$O = \bigJoin_{p \in [q]} A_p$.
For~$p \in [q]$ and~$I \in \II$, we have~$A_p(I) \in I$, hence~${O(I) \ge A_p(I)}$, so that~$O \ge A_p$ by \cref{prop:sourceOrderI}.
Consider now an acyclic orientation~$O'$ of~$\II$ such that~$A_p \le O'$ for all~$p \in [q]$.
Assume that there are~$I, J \in \II$ and~$p \in [q]$ such that~$A_p(J) \in I$.
Then~$A_p(J) \in I \cap J \in \II$, so that~$A_p(I \cap J) = A_p(J)$ by acyclicity of~$A_p$.
By \cref{prop:sourceOrderI}, we thus obtain~$O'(I \cap J) = A_p(J)$.
By acyclicity of~$O'$, this implies that~$O'(I) \notin {[\min(I \cap J), A_p(J)[}$, so that~$O'(I) \notin [\min(J), A_p(J)[$.
We conclude that~$O(I) \le O'(I)$.
Hence, $O \le O'$ by \cref{prop:sourceOrderI}.
\end{proof}


\section{Distributive interval hypergraphic lattices}
\label{sec:distributive}

\begin{figure}
	\centerline{
		\begin{tabular}{c@{\qquad}c}
			\begin{tikzpicture}[scale=2.5]
				\node (a) at (2,0) {\acyclicOrientation{4}{1/3/1,2/3/2,2/4/2}};
				\node (b) at (1,1) {\acyclicOrientation{4}{1/3/2,2/3/2,2/4/2}};
				\node (c) at (2,1) {\acyclicOrientation{4}{1/3/1,2/3/3,2/4/3}};
				\node (d) at (3,1) {\acyclicOrientation{4}{1/3/1,2/3/2,2/4/4}};
				\node (e) at (1,2) {\acyclicOrientation{4}{1/3/3,2/3/3,2/4/3}};
				\node (f) at (2,2) {\acyclicOrientation{4}{1/3/2,2/3/2,2/4/4}};
				\node (g) at (3,2) {\acyclicOrientation{4}{1/3/1,2/3/3,2/4/4}};
				\node (h) at (2,3) {\acyclicOrientation{4}{1/3/3,2/3/3,2/4/4}};
				\draw (a)--(b);
				\draw (a)--(c);
				\draw (a)--(d);
				\draw (b)--(e);
				\draw (b)--(f);
				\draw (c)--(e);
				\draw (c)--(g);
				\draw (d)--(f);
				\draw (d)--(g);
				\draw (e)--(h);
				\draw (f)--(h);
				\draw (g)--(h);
			\end{tikzpicture}
			&
			\begin{tikzpicture}[scale=2.5]
				\node (a) at (2,0) {\acyclicOrientation{4}{1/3/1,1/4/1}};
				\node (b) at (1,1) {\acyclicOrientation{4}{1/3/1,1/4/4}};
				\node (c) at (3,1) {\acyclicOrientation{4}{1/3/2,1/4/2}};
				\node (d) at (2,2) {\acyclicOrientation{4}{1/3/2,1/4/4}};
				\node (e) at (4,2) {\acyclicOrientation{4}{1/3/3,1/4/3}};
				\node (f) at (3,3) {\acyclicOrientation{4}{1/3/3,1/4/4}};
				\draw (a)--(b);
				\draw (a)--(c);
				\draw (b)--(d);
				\draw (c)--(d);
				\draw (c)--(e);
				\draw (d)--(f);
				\draw (e)--(f);
			\end{tikzpicture}
			\\[.2cm]
			$\{1,2,3,4,123,23,234\}$
			&
			$\{1,2,3,4,123,1234\}$
			\\
			boolean lattice
			&
			distributive lattice
			\\[.8cm]
			\begin{tikzpicture}[scale=2.5]
				\node (a) at (2,0) {\acyclicOrientation{4}{1/2/1,3/4/3,1/4/1}};
				\node (b) at (1,1) {\acyclicOrientation{4}{1/2/2,3/4/3,1/4/2}};
				\node (c) at (2,1) {\acyclicOrientation{4}{1/2/1,3/4/4,1/4/1}};
				\node (d) at (3,1) {\acyclicOrientation{4}{1/2/1,3/4/3,1/4/3}};
				\node (e) at (1,2) {\acyclicOrientation{4}{1/2/2,3/4/4,1/4/2}};
				\node (f) at (2,2) {\acyclicOrientation{4}{1/2/2,3/4/3,1/4/3}};
				\node (g) at (3,2) {\acyclicOrientation{4}{1/2/1,3/4/4,1/4/4}};
				\node (h) at (2,3) {\acyclicOrientation{4}{1/2/2,3/4/4,1/4/4}};
				\draw (a)--(b);
				\draw (a)--(c);
				\draw (a)--(d);
				\draw (b)--(e);
				\draw (b)--(f);
				\draw (c)--(e);
				\draw (c)--(g);
				\draw (d)--(f);
				\draw (d)--(g);
				\draw (e)--(h);
				\draw (f)--(h);
				\draw (g)--(h);
			\end{tikzpicture}
			&
			\begin{tikzpicture}[scale=2.5]
				\node (a) at (2,0) {\acyclicOrientation{4}{1/2/1,1/3/1,1/4/1}};
				\node (b) at (1,1) {\acyclicOrientation{4}{1/2/2,1/3/2,1/4/2}};
				\node (c) at (2,1) {\acyclicOrientation{4}{1/2/1,1/3/3,1/4/3}};
				\node (d) at (3,1) {\acyclicOrientation{4}{1/2/1,1/3/1,1/4/4}};
				\node (e) at (1,2) {\acyclicOrientation{4}{1/2/2,1/3/3,1/4/3}};
				\node (f) at (2,2) {\acyclicOrientation{4}{1/2/2,1/3/2,1/4/4}};
				\node (g) at (3,2) {\acyclicOrientation{4}{1/2/1,1/3/3,1/4/4}};
				\node (h) at (2,3) {\acyclicOrientation{4}{1/2/2,1/3/3,1/4/4}};
				\draw (a)--(b);
				\draw (a)--(c);
				\draw (a)--(d);
				\draw (b)--(e);
				\draw (b)--(f);
				\draw (c)--(e);
				\draw (c)--(g);
				\draw (d)--(f);
				\draw (d)--(g);
				\draw (e)--(h);
				\draw (f)--(h);
				\draw (g)--(h);
			\end{tikzpicture}
			\\[.2cm]
			$\{1,2,3,4,12,34,1234\}$
			&
			$\{1,2,3,4,12,123,1234\}$
			\\
			boolean lattice
			&
			boolean lattice
		\end{tabular}
	}
	\caption{Four distributive interval hypergraphic lattices.}
	\label{fig:distributiveLattices}
\end{figure}

In this section, we prove \cref{thm:distributiveLatticeI} which we first introduce properly:

\begin{definition}
\label{def:distributive}
We say that an interval hypergraph~$\II$ is \defn{distributive} if for all~$I, J \in \II$ such that~$I \not\subseteq J$, $I \not\supseteq J$ and~$I \cap J \ne \varnothing$, the intersection~$I \cap J$ is in~$\II$ and is initial or final in any~$K \in \II$ with~$I \cap J \subseteq K$.
\end{definition}

\begin{definition}
A lattice~$(L, \le , \join, \meet)$ is \defn{distributive} if
\[
a \join (b \meet c) = (a \join b) \meet (a \join c)
\qquad\text{and}\qquad
a \meet (b \join c) = (a \meet b) \join (a \meet c)
\]
for all~$a,b,c \in L$ (the two conditions are in fact equivalent).
Equivalently, $L$ is distributive if and only if it is isomorphic to the inclusion lattice on lower sets of its poset of join irreducible elements.
\end{definition}

\begin{theoremA}
For an interval hypergraph $\II$ on~$[n]$ (with our convention that~$\{i\} \in \II$ for all~$i \in [n]$), the poset $P_\II$ is a distributive lattice if and only if $\II$ is distributive.
\end{theoremA}

\begin{example}
The interval hypergraphic posets of \cref{fig:distributiveLattices} are distributive lattices, while those of \cref{fig:notLattices,fig:Tamari,fig:semidistributiveLattices,fig:notSemidistributiveLattices} are not.
\end{example}

\begin{example}
The hypergraphic poset of~$\{ 1, 2, 3, 12, 123\}$ is a distributive lattice (it is a diamond).
In general, the hypergraphic poset of all initial intervals is the boolean lattice, which is distributive.
See \cref{fig:distributiveLattices} and \cref{subsec:SchroderHypergraphs} for more examples of distributive hypergraphs.
\end{example}

\begin{example}
The hypergraphic poset of~$\{ 1, 2, 3, 12, 23, 123 \}$ is a semidistributive but not distributive lattice (it is a pentagon).
In general, the hypergraphic poset of all intervals is the Tamari lattice, which is semidistributive but not distributive.
See \cref{fig:Tamari}.
\end{example}


\subsection{Some join irreducible acyclic orientations}  
\label{subsec:someJoinIrreducibles}

In this section, we assume that~$\II$ is an interval hypergraph closed under intersection so that the hypergraph poset~$P_\II$ is a lattice  by \cref{prop:latticeBackward}.
Our first task will be to identify some join irreducible acyclic orientations of~$\II$ (a complete but more technical description of all join irreducible acyclic orientations of~$\II$ will appear later in~\cref{subsec:joinIrreducibles}).
The following notations are illustrated in \cref{exm:Aj1,exm:Tamari1,exm:Tamari2,exm:Aj4,exm:Aj5} below.

\begin{notation}
Define
\(\displaystyle
\cJ_\II \eqdef \bigcup_{I\in \II} I\ssm \{\min(I)\}.
\)
\end{notation}

\begin{notation}
\label{not:muj}
For~$j \in \cJ_\II$, we let
\[
J_j= \bigcap_{\substack{I \in \II \\ j \in I \ssm \{\min(I)\}}} I,
\]
(note that $J_j \in \II$ since $\II$ is closed under intersection), and we set~$\mu_j \eqdef \min(J_j)$ and~$\nu_j \eqdef \max(J_j)$.
By definition~$\min(I) \le \mu_j < j \le \nu_j \le \max(I)$ for any~$I \in \II$ such that~$j \in I \ssm \{\min(I)\}$.
\end{notation}

\begin{notation}
\label{not:Aj}
For~$j \in \cJ_\II$, consider the acyclic orientation~$A_j \eqdef \Or_{(\mu_j, \mu_j+1, \ldots, j)}$, obtained as the image by the surjection map~$\Or$ of \cref{def:surjection} of the cycle permutation
\[
(\mu_j, \mu_j+1, \dots, j) = 12 \cdots (\mu_j-1)j\,\mu_j \cdots (j-1) (j+1) \cdots n
\]
obtained from the identity permutation by placing~$j$ just before $\mu_j$.
Note that that for all $J \in \II$
\[
A_j(J) =
\begin{cases}
	j & \text{if } j \in J \text{ and } \min(J)=\mu_j,\\
	\min(J) & \text{otherwise.}
\end{cases}
\]
\end{notation}

\begin{notation}
For $i,j \in \cJ_\II$, we write $i \preccurlyeq j$ if and only if $J_i = J_j$ and $i \le j$.
\end{notation}

\begin{example}
\label{exm:Aj1}
For the hypergraph~$\II = \{1, 2, 3, 12, 123\}$, we have~$\cJ_\II = \{2,3\}$ and the corresponding acyclic orientations are
\[
A_2 = \acyclicOrientation{3}{1/2/2,1/3/2}
\qquad\text{and}\qquad
A_3 = \acyclicOrientation{3}{1/2/1,1/3/3}
\]
and are incomparable.
\end{example}

\begin{example}
\label{exm:Tamari1}
For the hypergraph~$\II = \{1, 2, 3, 12, 23, 123\}$  illustrated in \cref{fig:Tamari}\,(left), we have~$\cJ_\II = \{2,3\}$ and the corresponding acyclic orientations are
\[
A_2 = \acyclicOrientation{3}{1/2/2,2/3/2,1/3/2}
\qquad\text{and}\qquad
A_3 = \acyclicOrientation{3}{1/2/1,2/3/3,1/3/1}
\]
and are incomparable.
\end{example}

\begin{example}
\label{exm:Tamari2}
For the hypergraph~$\II \eqdef \{1, 2, 3, 12, 23, 34, 123, 234, 1234\}$ illustrated in \cref{fig:Tamari}\,(right), we have~$\cJ_\II = \{2,3,4\}$, and the corresponding acyclic orientations~are
\[
A_2 = \raisebox{-.4cm}{\acyclicOrientation{4}{1/2/2,2/3/2,3/4/3,1/3/2,2/4/2,1/4/2}}
\qquad
A_3 = \raisebox{-.4cm}{\acyclicOrientation{4}{1/2/1,2/3/3,3/4/3,1/3/1,2/4/3,1/4/1}}
\quad\text{and}\quad
A_4 = \raisebox{-.4cm}{\acyclicOrientation{4}{1/2/1,2/3/2,3/4/4,1/3/1,2/4/2,1/4/1}}
\]
and are incomparable.
\end{example}

\begin{example}
\label{exm:Aj4}
For the hypergraph~$\II \eqdef \{1,2,3,4,123,23,234\}$ illustrated in \cref{fig:distributiveLattices}\,(top left), we have~$\cJ_\II = \{2,3,4\}$, and the corresponding acyclic orientations~are
\[
A_2 = \acyclicOrientation{4}{1/3/2,2/3/2,2/4/2}
\qquad
A_3 = \acyclicOrientation{4}{1/3/1,2/3/3,2/4/3}
\quad\text{and}\quad
A_4 = \acyclicOrientation{4}{1/3/1,2/3/2,2/4/4}
\]
and are incomparable.
\end{example}

\begin{example}
\label{exm:Aj5}
For the hypergraph~$\II \eqdef \{1,2,3,4,123,1234\}$ illustrated in \cref{fig:distributiveLattices}\,(top right), we have~$\cJ_\II = \{2,3,4\}$, and the corresponding acyclic orientations~are
\[
A_2 = \acyclicOrientation{4}{1/3/2,1/4/2}
\qquad
A_3 = \acyclicOrientation{4}{1/3/3,1/4/3}
\quad\text{and}\quad
A_4 = \acyclicOrientation{4}{1/3/1,1/4/4}
\]
and we have~$A_2 < A_3$.
\end{example}

\begin{lemma}
\label{lem:distinctIrreducibles}
For $i \ne j \in \cJ_\II$, we have $A_i \ne A_j$.
\end{lemma}

\begin{proof}
If~$i \ne j$, then we have~$A_i(J_j) \in \{\mu_j, i\}$ while~$A_j(J_j) = j \notin \{\mu_j,i\}$.
\end{proof}

\begin{lemma}
\label{lem:irrorder}
For $i,j \in \cJ_\II$, we have $A_i \le A_j \iff i \preccurlyeq j$.
\end{lemma}

\begin{proof}
By \cref{prop:sourceOrderI} we have~$A_i \le A_j$ if and only if~$A_i(I) \le A_j(I)$ for all~$I \in \II$.
The forward direction is thus a direct consequence of the following four observations:
\begin{itemize}
\item if~$\mu_i \ne \mu_j$, then~$A_i(J_i) = i > \mu_i = \min(J_i) = A_j(J_i)$,
\item if~$\nu_i < \nu_j$, then~$j \notin J_i \ssm \{\mu_i\}$, so that~$A_i(J_i) = i > \mu_i = \min(J_i) = A_j(J_i)$,
\item if~$i > j$, then~$A_i(J_i) = i > j \ge A_j(J_i)$,
\item $i \le j$ implies that~$\mu_i \le \mu_j$ and~$\nu_i \le \nu_j$.
\end{itemize}
For the backward direction, assume~$i \preccurlyeq j $, and consider~$I \in \II$.
As~$i < j$, $A_i(I) \in \{\min(I), i\}$ and~$A_j(I) \in \{\min(I), j\}$, we have~$A_i(I) \le A_j(I)$ except if~$A_i(I) = i$ and~$A_j(I) = \min(I)$.
As~$A_i(I) = i$, we would have~$i \in I$ and~$\min(I) = \mu_i$. As~$A_j(I) = \min(I)$ and~$\min(I) = \mu_i = \mu_j$, we would have~$j \notin I$.
This contradicts the fact that~$\nu_i = \nu_j$.
\end{proof}

\begin{lemma}
\label{lem:subirr}
For any~$j \in \cJ_\II$ and any acyclic orientation~$A$ of~$\II$, we have
\begin{itemize}
\item $A \le A_j \iff A = \min(P_\II)$ or~$A = A_i$ with~$i \preccurlyeq j$,
\item $A_j \le A \iff j \le A(J_j)$.
\end{itemize}
\end{lemma}

\begin{proof}
By \cref{prop:sourceOrderI}, $A \le A_j$ implies that~$A(I) \le A_j(I)$ for all~$I \in \II$.
For~$I, I' \in \II$ such that~$j \in I \cap I'$ and~$\min(I) = \min(I') = \mu_j$, we have~$A(I) = A(I')$ since~$A$ is acyclic.
If~$A < A_j$, we conclude that there is~$i \in {[\mu_j, j[}$ such that~$A(I) = i$ if~$j \in I$ and~$\min(I) = \mu_j$, and~$A(I) = \min(I)$ otherwise.
Hence, $A = \min(P_\II)$ or~$A = A_i$.
The first point thus follows from \cref{lem:irrorder}.

For the second point, $A_j \le A$ implies $j = A_j(J_j) \le A(J_j)$ by \cref{prop:sourceOrderI}.
Conversely, if~$j \le A(J_j)$, then for any~$I \in \II$,
\begin{itemize}
\item if~$j \in I$ and~$\min(I) = \mu_j$, we have~$A_j(I) = j \le A(I)$ (by acyclicity of~$A$),
\item otherwise, $A_j(I) = \min(I) \le A(I)$.
\end{itemize}
Hence, $A_j \le A$ by \cref{prop:sourceOrderI}.
\end{proof}

\begin{proposition}
\label{prop:AjJoinIrreducible}
If~$\II$ is an interval hypergraph closed under intersections, then the acyclic orientation~$A_j$ is join irreducible for any~$j\in \cJ_\II$.
\end{proposition}

\begin{proof}
By \cref{lem:subirr}, the lower set of~$A_j$ in~$P_\II$ is the chain $\min(P_\II) < A_{i_1} < \dots < A_{i_p} < A_j$ where~$\{i_1 < \dots < i_p\} = \set{i \in \cJ_\II}{J_i = J_j \text{ and } i < j}$.
Hence, $A_j$ is join irreducible.
\end{proof}

\begin{corollary}
\label{coro:irreduciblePosetMorphism}
If~$\II$ is an interval hypergraph closed under intersections, then the map $j \mapsto A_j$ is an injective poset morphism from~$(\cJ_\II, \preccurlyeq)$ to the subposet of join irreducibles of~$P_\II$.
\end{corollary}

\begin{proof}
Immediate from \cref{lem:distinctIrreducibles,prop:AjJoinIrreducible,lem:irrorder}.
\end{proof}

In the next two propositions, we assume $\II$ is closed under intersections.

\begin{proposition}
\label{prop:injectivityDistributive}
For any two lower sets~$X$ and~$Y$ of~$(\cJ_\II, \preccurlyeq)$, if~$\bigJoin\limits_{x \in X} A_x = \bigJoin\limits_{y \in Y} A_y$ then~$X = Y$.
\end{proposition}

\begin{proof}
Assume~$\bigJoin\limits_{x \in X} A_x = \bigJoin\limits_{y \in Y} A_y$ and let~$x \in X$.
Then $x = A_x(J_x) \le \big( \bigJoin\limits_{x \in X} A_x \big)(J_x) = \big( \bigJoin\limits_{y \in Y} A_y \big)(J_x)$.
From the description of the join of \cref{prop:joinLattice}, we thus obtain that there exists~$y \in Y$ and~$J \in \II$ such that~$A_y(J) \in J_x$ and~$\min(J) < x \le A_y(J)$.
As~$\min(J) \ne A_y(J)$, we obtain that~$A_y(J) = y$ and~$\min(J) = \mu_y$.
We get that~$\mu_y < x \le y \le \nu_x$, hence that~$x \in J_y \ssm \{\mu_y\}$ and~$y \in J_x \ssm \{\mu_x\}$ so that~$J_x = J_y$.
As~$x \le y$, we conclude that~$x \preccurlyeq y$, so that~$x \in Y$.
We thus obtained that~$X \subseteq Y$, and thus~$X = Y$ by symmetry.
\end{proof}

\begin{proposition}
$\displaystyle \max(P_\II) = \bigJoin\limits_{j \in \cJ_\II} A_j$.
\end{proposition}

\begin{proof}
Assume by contradiction~$\max(P_\II) \ne \bigJoin_{j \in \cJ_\II} A_j \defeq A$.
Then there is~$I \in \II$ with~${A(I) < \max(I)}$.
Let~$j \eqdef A(I) + 1$.
As~$j \in I \ssm \{\min(I)\}$, we have~$A(J_j) \in J_j \subseteq I$ and~$A(I) = j-1 \in J_j$.
As~$A_j \le A$, we have~$A(I) = j-1 < j = A_j(J_j) \le A(J_j)$ so that~$A(I) \ne A(J_j)$.
We thus obtain that~$A$ is cyclic, a contradiction.
\end{proof}

\begin{corollary}
For any interval hypergraph~$\II$ closed under intersection, $P_\II$ contains a distributive sublattice containing~$\min(P_\II)$ and~$\max(P_\II)$.
\end{corollary}


\subsection{If~$\II$ is distributive then $P_\II$ is distributive}
\label{subsec:distributiveLatticeBackward}

We now prove the backward direction of \cref{thm:distributiveLatticeI}.
The following strengthen \cref{prop:alwaysFlippableI}.

\begin{proposition}
\label{prop:alwaysFlippableDistributive}
For a distributive interval hypergraph~$\II$, an acyclic orientation~$A$ of~$\II$, and~${J \in \II}$ such that~$j \eqdef A(J) \ne \min(J)$, there exists~$i$ such that
\[
\min(J) \le i \le \max\set{\max(I)}{I \in \II, \; \min(I) = \min(J), \; \max(I) < j}
\]
and the orientation obtained from~$A$ by flipping~$j$ to~$i$ is acyclic.
\end{proposition}

\begin{proof}
Let~$I \in \II$ be maximal such that~$\min(I) = \min(J)$ and~$\max(I) < j$, and let~$i \eqdef A(I)$.
Let~$O$ be the orientation obtained from~$A$ by flipping~$j$ to~$i$ as in \cref{def:flip}.
If~$O$ is cyclic, there exists~$K,K' \in \II$ such that~$O(K) \in K'$, $O(K') \in K$, and~$O(K) \ne O(K')$.
As~$A$ is acyclic, we have either~$A(K) \ne O(K)$ or~$A(K') \ne O(K')$, but not both since~$O(K) \ne O(K')$.
We can thus assume that~$A(K) = O(K)$ while $A(K') = j$ and~$O(K') = i$.
If~$j \in K$, then we have~$A(K) \ne j = A(K')$ (otherwise, $O(K) = i = O(K')$), and $A(K) = O(K) \in K'$ and~$A(K') = j \in K$, contradicting the acyclicity of~$A$.
As~$i \in K$ and~$j \notin K$, we obtain that~$\max(K) < j$.
As~$A(I) = i \in K$ and~$A(K) = O(K) \ne O(K') = i = A(I)$, we have~$K \not\subseteq I$ by acyclicity of~$A$.
If~$\min(I) < \min(K)$, then~$I \not\subseteq K$ and~$I \not\supseteq K$ and~$I \cap K \ni i$, and $K$ is neither initial nor final in~$J$, contradicting the distributivity of~$\II$.
If~$\min(K) < \min(J)$ and~$\min(K') < \min(J)$, then we have~$J \not\subseteq K$ and~$J \not\supseteq K$ and~$J \cap K \ni i$, and~$J \cap K$ is neither initial nor final in~$K'$, contradicting the distributivity of~$\II$.
As~$\II$ is closed under intersection and we have~$A(K) = O(K) \in K'$, we can thus assume that~$\min(I) = \min(K)$.
We thus obtain that~$K \subseteq I$.
As~$A(K) = O(K) \ne O(K') = i = A(I)$, this contradicts the acyclicity of~$A$.
We conclude that~$O$ is acyclic, which proves the statement.
\end{proof}

\begin{remark}
\label{rem:alwaysFlippable}
\cref{prop:alwaysFlippableDistributive} fails when~$\II$ is not distributive.
For instance, for the interval hypergraph~$\II = \{1, 2, 3, 12, 23, 123\}$ of \cref{exm:Tamari1,fig:Tamari}\,(left), for the interval~$J = 23$ and for the acyclic orientation
\[
A = \acyclicOrientation{3}{1/2/1,2/3/3,1/3/3}
\]
\end{remark}

\begin{proposition}
\label{prop:irreduciblePosetIsomorphism}
If~$\II$ is a distributive interval hypergraph, then the map $j \mapsto A_j$ is a poset isomorphism from~$(\cJ_\II, \preccurlyeq)$ to the subposet of join irreducibles of~$P_\II$.
\end{proposition}

\begin{proof}
We have already seen in \cref{coro:irreduciblePosetMorphism} that~$j \mapsto A_j$ is an injective poset morphism.
We thus just need to show that the distributivity of~$\II$ implies the surjectivity of this morphism.
Consider an increasing flip~\flip{A}{i}{j}{B}.
If there is~$J \in \II$ such that~$B(J) \notin \{\min(J), j\}$, then~$A(J) = B(J)$ and \cref{prop:alwaysFlippableI} ensures that $B$ admits a decreasing flip flipping~$B(J)$ to some~$k < B(J) = A(J)$.
If there is~$J \in \II$ such that~$B(J) = j$ and~$\min(J) < \mu_j$, then the distributivity of~$\II$ implies that~$\max(I) < \mu_j$ for any~$I \in \II$ such that~$\min(I) = \min(J)$ and~$\max(I) < j$ (otherwise, $I \not\subseteq J_j$ and~$I \not\supseteq J_j$ and~$\mu_j \in I \cap J_i \subseteq J$, and~$I \cap J_j$ is neither initial nor final in~$J$, contradicting the distributivity of~$\II$).
Hence, \cref{prop:alwaysFlippableDistributive} ensures that $B$ admits a decreasing flip flipping~$j$ to some~$k < \mu_j$.
As $B$ is acyclic, there is no~$J \in \II$ such that~$B(J) = \min(J) = \mu_j$ and~$j \in J$.
We conclude that if~$B$ admits a single decreasing flip, then~$B(J) = j$ if~$j \in J$ and~$\min(J) = \mu_j$, and~$B(J) = \min(J)$ otherwise, so that~$B = A_j$.
\end{proof}

\begin{remark}
We will see in \cref{lem:oneMoreIrreducible} that the surjectivity in \cref{prop:irreduciblePosetIsomorphism} systematically fails when~$\II$ is not distributive.
See \cref{exm:oneMoreIrreducible} for an example.
\end{remark}

\begin{proposition}
\label{prop:distributiveForwardI}
If~$\II$ is a distributive interval hypergraph, then the maps
\[
\Phi : A \mapsto \set{j \in \cJ_\II}{A_j \le A}
\qquad\text{and}\qquad
\Psi : X \mapsto \bigJoin_{x \in X} A_x
\]
are inverse bijections between the acyclic orientations of~$\II$ and the lower sets of~$(\cJ_\II, \preccurlyeq)$. Hence~$P_\II$ is a distributive lattice.
\end{proposition}

\begin{proof}
In a finite lattice, any element can always be written as the join of the join irreducible elements below it.
\cref{prop:irreduciblePosetIsomorphism} thus implies that~$\Psi(\Phi(A)) = A$.
The statement follows since $\Psi$ is injective by \cref{prop:injectivityDistributive}.
\end{proof}

\begin{corollary}
\label{coro:productChains}
If~$\II$ is a distributive hypergraph, then~$P_\II$ is a Cartesian product of chains.
\end{corollary}

\begin{proof}
The poset~$(\cJ_\II, \preccurlyeq)$ is a disjoint union of chains, hence its lattice of lower sets is a Cartesian product of chains.
\end{proof}

\begin{remark}
\label{rem:fhpolynomials}
It immediately follows from \cref{coro:productChains} that the $f$-polynomial and the $h$-polynomial of~$P_\II$ are given by
\[
f_{P_\II}(x) = \prod_{i \in [k]} \frac{(x+1)^{\ell_i+1}-1}{x}
\qquad\text{and}\qquad
g_{P_\II}(x) = \prod_{i \in [k]} \frac{x^{\ell_i+1}-1}{x-1}
\]
where~$\ell_1, \dots, \ell_k$ denote the length (number of elements) of the chains of~$(\cJ_\II, \preccurlyeq)$.
\end{remark}


\subsection{If~$P_\II$ is distributive then~$\II$ is distributive}
\label{subsec:distributiveLatticeForward}

We now prove the forward direction of \cref{thm:distributiveLatticeI}.
We first show that the surjectivity in \cref{prop:irreduciblePosetIsomorphism} systematically fails when~$\II$ is not distributive.

\begin{lemma}
\label{lem:oneMoreIrreducible}
If an interval hypergraph~$\II$ is closed under intersection but not distributive, there is a join irreducible acyclic orientation~$A$ such that~$A \not\le A_j$ for all~$j \in \cJ_\II$.
\end{lemma}

\begin{proof}
As $\II$ is closed under intersection but not distributive, there are~$I,J,K \in \II$ with~$I \not\subseteq J$,~$I \not\supseteq J$, $\varnothing \ne I \cap J \subseteq K$, and~$I \cap J$ is neither initial nor final in~$K$.
By symmetry, we can assume that~$\min(I) < \min(J) \le \max(I) < \max(J)$. 
As~$I \cap J$ is neither initial nor final in~$K$, we have~$\min(K) < \min(J) \le \max(I) < \max(K)$.
As~$\II$ is closed under intersection, we can even assume that~$\min(I) = \min(K)$ and~$\max(J) = \max(K)$.
Let~$i \eqdef \min(I) = \min(K)$ and~$j \eqdef \min(J \ssm I)$.
Let~$A \eqdef \Or_{(i, \dots, j)}$ be the acyclic orientation of~$\II$ obtained as the image by the surjection map~$\Or$ of \cref{def:surjection} of the cycle permutation~$(i, \dots, j)  = 12 \cdots (i-1)j\, i \cdots (j-1) (j+1) \cdots n$.
In other words,
\[
A(K) = 
\begin{cases}
j & \text{if } j \in J \text{ and } \min(J) \ge i \\
\min(J) & \text{otherwise.}
\end{cases}
\]
We claim that the flip of~$j$ to~$i$ is the only decreasing flip from~$A$.
Let~$B$ be an acyclic orientation of~$\II$ obtained by a decreasing flip from~$A$.
Let~$H \in \II$ be such that~$B(H) < A(H)$.
Since~$A(H) \in \{\min(H), j\}$, we have~$A(H) = j$, so that~$i \le \min(H) \le B(H) < A(H) = j$.
If~$B(H) \ne i$, then~$B(K) = B(H) \in I$ (since~$B(H) \in K$ and~$A(K) = j$) and $B(I) \le A(I) = \min(I) = i \in K$ (since~$j \notin I$), contradicting the acyclicity of~$B$.
We conclude that~$B(H) = i$, so that~$B$ is indeed obtained from~$A$ by flipping $j$ to~$i$.

Finally, as~$A(K) = j$ while~$A_j(K) = \min(K) < j$, we have~$A \not\le A_j$ for all~$j \in \cJ_\II$ by \cref{prop:sourceOrderI}.
\end{proof}

\begin{example}
\label{exm:oneMoreIrreducible}
Following on \cref{rem:alwaysFlippable}, the join irreducible acyclic orientations of the interval hypergraph~$\II = \{1, 2, 3, 12, 23, 123\}$ of \cref{exm:Tamari1,fig:Tamari}\,(left) are
\[
A_2 = \acyclicOrientation{3}{1/2/2,2/3/2,1/3/2}
\qquad\quad
A_3 = \acyclicOrientation{3}{1/2/1,2/3/3,1/3/1}
\qquad\text{and}\qquad
A = \acyclicOrientation{3}{1/2/1,2/3/3,1/3/3}
\]
\end{example}

\begin{proposition}
If~$\II$ is an interval hypergraph such that $P_\II$ is a distributive lattice, then~$\II$ is distributive.
\end{proposition}

\begin{proof}
If~$\II$ is not closed under intersection, then~$P_\II$ is not even a lattice by \cref{prop:latticeForward}.
If~$\II$ is closed under intersection but not distributive, consider the join irreducible acyclic orientation~$A$ of \cref{lem:oneMoreIrreducible}.
As~$A \not\le A_j$ for all~$j \in \cJ_\II$, and~$\max(P_\II) = \bigJoin_{j \in \cJ_\II} A_j = A \join \bigJoin_{j \in \cJ_\II} A_j$, we obtain that~$P_\II$ is not distributive.
\end{proof}

\begin{corollary}
For an internal hypergraph~$\II$, the poset~$P_\II$ is a distributive lattice if and only if it is a Cartesian product of chains.
\end{corollary}


\subsection{Schr\"oder hypergraphs}
\label{subsec:SchroderHypergraphs}

As an illustration of this section, we now consider a special family of distributive hypergraphs which were already considered in~\cite[Sect.~7.2]{PostnikovReinerWilliams} and~\cite{Defant-fertilitopes}.

\begin{definition}
A hypergraph~$\HH$ is \defn{laminar} if any two hyperedges~$G, H \in \HH$ are either disjoint ($G \cap H = \varnothing$) or nested ($G \subseteq H$ or~$G \supseteq H$).
\end{definition}

\begin{definition}
Let~$S$ be a \defn{Schr\"oder tree} (\ie a rooted plane tree where each internal node has at least two children) with~$n$ leaves labeled by~$[n]$ from left to right.
Label each node of~$S$ by the set of leaves of its subtree.
We say that the hypergraph~$\II_S$ formed by all singletons and the labels of the nodes of~$S$ is a \defn{Schr\"oder hypergraph}.
\end{definition}

\begin{example}
All but the top left interval hypergraphs of \cref{fig:distributiveLattices} are Schr\"oder hypergraphs. 
The corresponding Schr\"oder trees are given by:
\[
	\begin{array}{c|ccc}
		\raisebox{1cm}{$S$}
		& 
		\tree{[.1234 [.123 [.{\phantom{1}} ] [.{\phantom{1}} ] [.{\phantom{1}} ] ] [.{\phantom{1}} ] ]}
		& 
		\tree{[.1234 [.12 [.{\phantom{1}} ] [.{\phantom{1}} ] ] [.34 [.{\phantom{1}} ] [.{\phantom{1}} ] ] ]}
		& 
		\tree{[.1234 [.123 [.12 [.{\phantom{1}} ] [.{\phantom{1}} ] ] [.{\phantom{1}} ] ] [.{\phantom{1}} ] ]}
		\\[-.3cm]
		\hline
		\\[-.3cm]
		\II_S
		&
		\{1,2,3,4,123,1234\}
		&
		\{1,2,3,4,12,34,1234\}
		&
		\{1,2,3,4,12,123,1234\}
	\end{array}
\]
\end{example}

\begin{example}
For instance, for the left comb~$C$, the Schr\"oder hypergraph~$\II_C$ consists of all singletons and all initial intervals. The hypergraphic polytope~$\simplex_{\II_C}$ is the Pitman--Stanley polytope~\cite{PitmanStanley}.
\end{example}

\begin{proposition}
\label{prop:characterizationSchroder}
The Schr\"oder hypergraphs are precisely the laminar interval hypergraphs.
\end{proposition}

\begin{proof}
Consider first two nodes~$i$ and~$j$ of a Schr\"oder tree~$S$, and let~$I$ and~$J$ be the corresponding labels.
Then~$I \subseteq J$ if~$i$ is a descendant of~$j$, $I \supseteq J$ if~$i$ is a ancestor of~$j$, and~$I \cap J = \varnothing$ otherwise.
Hence, any Schr\"oder hypergraph is a laminar interval hypergraph.

Conversely, consider a laminal interval hypergraph~$\II$. Then the inclusion poset on~$\II$ is a Schr\"oder tree~$S$ with~$\II_S = \II$.
\end{proof}

\begin{remark}
\cref{prop:characterizationSchroder} implies that Schr\"oder hypergraphs are building sets, so that Schr\"oder hypergraphic polytopes are specific nestohedra~\cite{FeichtnerSturmfels,Postnikov}.
In fact, they were already considered in~\cite[Sect.~7.2]{PostnikovReinerWilliams} and~\cite{Defant-fertilitopes}.
\end{remark}

\begin{corollary}
Any Schr\"oder hypergraph is a distributive interval hypergraph, hence the Schr\"oder hypergraphic posets are distributive lattices.
\end{corollary}

\begin{proof}
The distributivity condition of \cref{def:distributive} is clearly fulfilled as any~$I \subseteq J$ or~$I \supseteq J$ or~$I \cap J = \varnothing$ for all~$I,J \in \II$.
\end{proof}

\begin{proposition}
\label{prop:schroder}
Given a Schr\"oder tree~$S$, the poset of join irreducible acyclic orientations on~$\II_S$ is isomorphic to a disjoint union of chains.
More precisely, it has one chain for each node~$I$ of~$S$, whose elements are the leaves of~$I$ and the leftmost leaves of the children of~$I$, except the leftmost leaf below~$I$.
\end{proposition}

\begin{proof}
This description is a specialization of the description of the join irreducible poset of distributive interval hypergraphic posets from \cref{subsec:someJoinIrreducibles}.
\end{proof}

\begin{remark}
Similarly to \cref{rem:fhpolynomials}, \cref{prop:schroder} implies that the $f$-polynomial and the $h$-polynomial of~$P_{\II_S}$ are given by
\[
f_{P_{\II_S}}(x) = \prod_{I \in S} \frac{(x+1)^{\deg(I)}-1}{x}
\qquad\text{and}\qquad
g_{P_{\II_S}}(x) = \prod_{I \in S} \frac{x^{\deg(I)}-1}{x-1}
\]
where~$I$ runs over the nodes of~$S$ and~$\deg(I)$ denotes its degree in~$S$.
\end{remark}


\section{Semidistributive interval hypergraphic lattices}
\label{sec:semidistributive}

\afterpage{
\begin{figure}
	\centerline{
		\begin{tabular}{c@{\qquad}c}
			\begin{tikzpicture}[scale=2.5]
				\node (a) at (3,0) {\acyclicOrientation{4}{1/3/1,2/3/2,2/4/2,1/4/1}};
				\node (b) at (1,1) {\acyclicOrientation{4}{1/3/2,2/3/2,2/4/2,1/4/2}};
				\node (c) at (3,1) {\acyclicOrientation{4}{1/3/1,2/3/3,2/4/3,1/4/1}};
				\node (d) at (4,1) {\acyclicOrientation{4}{1/3/1,2/3/2,2/4/4,1/4/1}};
				\node (e) at (1,2) {\acyclicOrientation{4}{1/3/3,2/3/3,2/4/3,1/4/3}};
				\node (f) at (3,2) {\acyclicOrientation{4}{1/3/1,2/3/2,2/4/4,1/4/4}};
				\node (g) at (4,2) {\acyclicOrientation{4}{1/3/1,2/3/3,2/4/4,1/4/1}};
				\node (h) at (2,3) {\acyclicOrientation{4}{1/3/2,2/3/2,2/4/4,1/4/4}};
				\node (i) at (3,3) {\acyclicOrientation{4}{1/3/1,2/3/3,2/4/4,1/4/4}};
				\node (j) at (2,4) {\acyclicOrientation{4}{1/3/3,2/3/3,2/4/4,1/4/4}};
				\draw (a)--(b);
				\draw (a)--(c);
				\draw (a)--(d);
				\draw (b)--(e);
				\draw (b)--(h);
				\draw (c)--(e);
				\draw (c)--(g);
				\draw (d)--(f);
				\draw (d)--(g);
				\draw (e)--(j);
				\draw (f)--(h);
				\draw (f)--(i);
				\draw (g)--(i);
				\draw (h)--(j);
				\draw (i)--(j);
			\end{tikzpicture}
			&
			\begin{tikzpicture}[scale=2.5]
				\node (a) at (2,0) {\acyclicOrientation{4}{2/3/2,3/4/3,1/4/1}};
				\node (b) at (1,1) {\acyclicOrientation{4}{2/3/2,3/4/3,1/4/2}};
				\node (c) at (2,1) {\acyclicOrientation{4}{2/3/3,3/4/3,1/4/1}};
				\node (d) at (3,1) {\acyclicOrientation{4}{2/3/2,3/4/4,1/4/1}};
				\node (e) at (1,2) {\acyclicOrientation{4}{2/3/3,3/4/3,1/4/3}};
				\node (f) at (2,2) {\acyclicOrientation{4}{2/3/2,3/4/4,1/4/2}};
				\node (g) at (3,2) {\acyclicOrientation{4}{2/3/3,3/4/4,1/4/1}};
				\node (h) at (2,3) {\acyclicOrientation{4}{2/3/2,3/4/4,1/4/4}};
				\node (i) at (2,4) {\acyclicOrientation{4}{2/3/3,3/4/4,1/4/4}};
				\draw (a)--(b);
				\draw (a)--(c);
				\draw (a)--(d);
				\draw (b)--(e);
				\draw (b)--(f);
				\draw (c)--(e);
				\draw (c)--(g);
				\draw (d)--(f);
				\draw (d)--(g);
				\draw (e)--(i);
				\draw (f)--(h);
				\draw (g)--(i);
				\draw (h)--(i);
			\end{tikzpicture}
			\\[.2cm]
			$\{1,2,3,4,123,23,234,1234\}$
			&
			$\{1,2,3,4,23,34,1234\}$
			\\
			semidistributive lattice,
			&
			semidistributive lattice,
			\\
			but not distributive
			&
			but not distributive
		\end{tabular}
	}
	\caption{Two interval hypergraphic lattices which are semidistributive but not distributive.}
	\label{fig:semidistributiveLattices}
\end{figure}
\begin{figure}
	\centerline{
		\begin{tabular}{c@{\qquad}c}
			\begin{tikzpicture}[scale=2.5]
				\node (a) at (3,0) {\acyclicOrientation{4}{1/2/1,2/3/2,3/4/3,1/4/1}};
				\node (b) at (1,1) {\acyclicOrientation{4}{1/2/1,2/3/3,3/4/3,1/4/1}};
				\node (c) at (3,1) {\acyclicOrientation{4}{1/2/1,2/3/2,3/4/4,1/4/1}};
				\node (d) at (4,1) {\acyclicOrientation{4}{1/2/2,2/3/2,3/4/3,1/4/2}};
				\node (e) at (1,2) {\acyclicOrientation{4}{1/2/1,2/3/3,3/4/4,1/4/1}};
				\node (f) at (2,2) {\acyclicOrientation{4}{1/2/1,2/3/3,3/4/3,1/4/3}};
				\node (g) at (3,2) {\acyclicOrientation{4}{1/2/1,2/3/2,3/4/4,1/4/4}};
				\node (h) at (4,2) {\acyclicOrientation{4}{1/2/2,2/3/2,3/4/4,1/4/2}};
				\node (i) at (1,3) {\acyclicOrientation{4}{1/2/1,2/3/3,3/4/4,1/4/4}};
				\node (j) at (3,3) {\acyclicOrientation{4}{1/2/2,2/3/3,3/4/3,1/4/3}};
				\node (k) at (4,3) {\acyclicOrientation{4}{1/2/2,2/3/2,3/4/4,1/4/4}};
				\node (l) at (3,4) {\acyclicOrientation{4}{1/2/2,2/3/3,3/4/4,1/4/4}};
				\draw (a)--(b);
				\draw (a)--(c);
				\draw (a)--(d);
				\draw (b)--(e);
				\draw (b)--(f);
				\draw (c)--(e);
				\draw (c)--(g);
				\draw (c)--(h);
				\draw (d)--(j);
				\draw (d)--(h);
				\draw (e)--(i);
				\draw (f)--(i);
				\draw (f)--(j);
				\draw (g)--(i);
				\draw (g)--(k);
				\draw (h)--(k);
				\draw (i)--(l);
				\draw (j)--(l);
				\draw (k)--(l);
			\end{tikzpicture}
			&
			\begin{tikzpicture}[scale=2.5]
				\node (a) at (2,0) {\acyclicOrientation{4}{1/2/1,2/3/2,2/4/2,1/4/1}};
				\node (b) at (1,1) {\acyclicOrientation{4}{1/2/1,2/3/3,2/4/3,1/4/1}};
				\node (c) at (2,1) {\acyclicOrientation{4}{1/2/2,2/3/2,2/4/2,1/4/2}};
				\node (d) at (3,1) {\acyclicOrientation{4}{1/2/1,2/3/2,2/4/4,1/4/1}};
				\node (e) at (1,2) {\acyclicOrientation{4}{1/2/1,2/3/3,2/4/3,1/4/3}};
				\node (f) at (2,2) {\acyclicOrientation{4}{1/2/1,2/3/3,2/4/4,1/4/1}};
				\node (g) at (3,2) {\acyclicOrientation{4}{1/2/1,2/3/2,2/4/4,1/4/4}};
				\node (h) at (1,3) {\acyclicOrientation{4}{1/2/2,2/3/3,2/4/3,1/4/3}};
				\node (i) at (2,3) {\acyclicOrientation{4}{1/2/1,2/3/3,2/4/4,1/4/4}};
				\node (j) at (3,3) {\acyclicOrientation{4}{1/2/2,2/3/2,2/4/4,1/4/4}};
				\node (k) at (2,4) {\acyclicOrientation{4}{1/2/2,2/3/3,2/4/4,1/4/4}};
				\draw (a)--(b);
				\draw (a)--(c);
				\draw (a)--(d);
				\draw (b)--(e);
				\draw (b)--(f);
				\draw (c)--(h);
				\draw (c)--(j);
				\draw (d)--(f);
				\draw (d)--(g);
				\draw (e)--(h);
				\draw (e)--(i);
				\draw (f)--(i);
				\draw (g)--(i);
				\draw (g)--(j);
				\draw (h)--(k);
				\draw (i)--(k);
				\draw (j)--(k);
			\end{tikzpicture}
			\\[.2cm]
			$\{1,2,3,4,12,23,34,1234\}$
			&
			$\{1,2,3,4,12,23,234,1234\}$
			\\
			lattice but neither join
			&
			meet semidistributive lattice,
			\\
			nor meet semidistributive
			&
			but not join semidistributive
		\end{tabular}
	}
	\caption{Two interval hypergraphic lattices which are not semidistributive.}
	\label{fig:notSemidistributiveLattices}
\end{figure}
}

In this section, we prove \cref{thm:semidistributiveLatticeI} which we first introduce properly:

\begin{definition}
\label{def:semidistributive}
We say that an interval hypergraph~$\II$ is \defn{join semidistributive} if it is closed under intersection and
for all~$[r,r'], [s,s'], [t,t'], [u,u'] \in \II$ such that ${r < s \le r' < s'}$, $r < t \le s' < t'$, $u < \min(s, t)$ and~$s' < u'$, there is~$[v,v'] \in \II$ such that~$v < s$ and~${s' < v' < t'}$.
\end{definition}

\begin{definition}
A lattice~$(L, \le , \join, \meet)$ is \defn{join semidistributive} if
\[
a \join b = a \join c
\qquad\text{implies}\qquad
a \join (b \meet c) = a \join b
\]
for all~$a,b,c \in L$.
Equivalently, $L$ is join semidistributive if and only if for any cover relation~$a \lessdot b$ in~$L$, the set~$\set{c \in L}{a \join c = b}$ admits a unique minimal element~$k_{a \lessdot b}$.
The meet semidistributivity is defined dually.
A lattice~$L$ is \defn{semidistributive} if it is both meet and join semidistributive.
\end{definition}

\begin{remark}
\label{rem:semidistributiveCriterion}
In general, the set~$\set{c \in L}{a \join c = b}$ might have more than one minimal element.
Note however that any minimal element of~$\set{c \in L}{a \join c = b}$ is always join irreducible.
Indeed, Assume that~$a \join c = b$ and~$c = d \join e$.
Then~$a \ne a \join d$ or~$a \ne a \join e$ since~$a < b = a \join c = a \join (d \join e) = (a \join d) \join (a \join e)$.
Moreover, $a \join d \le b$ and~$a \join e \le b$ since~$a \le b$ and~$d \le c \le b$ and~$e \le c \le b$.
As~$a \lessdot b$ is a cover relation of~$L$, we conclude that $a \join d = b$ or~$a \join e = b$ so that~$c$ is not minimal.
\end{remark}

\begin{remark}
Although we will not insist on this aspect in this paper, the join semidistributivity is equivalent to the existence of canonical join representations.
A \defn{join representation} of~${b \in L}$ is a subset~$J \subseteq L$ such that~$b = \bigJoin J$.
Such a representation is \defn{irredundant} if~$b \ne \bigJoin J'$ for any strict subset~$J' \subsetneq J$.
The irredundant join representations of~$b \in L$ are ordered by~$J \le J'$ if and only if for any~$j \in J$ there is~${j' \in J'}$~with~$j \le j'$.
The \defn{canonical join representation} of~$b$ is the minimal irredundant join representation of~$b$ for this order when it exists.
The lattice~$L$ is semidistributive if and only if any element of~$L$ admits a canonical join representation.
Moreover, the canonical join representation of~$b$ is given by~$b = \bigJoin_{a \lessdot b} k_{a \lessdot b}$, where $a$ ranges over all elements of~$L$ covered by~$b$.
\end{remark}

\begin{theoremA}
For an interval hypergraph $\II$ on~$[n]$ (with our convention that~$\{i\} \in \II$ for all~$i \in [n]$), the poset $P_\II$ is a join semidistributive lattice if and only if $\II$ is join semidistributive.
Under the symmetry of \cref{prop:antiIsomorphism}, a dual characterization holds for meet semidistributive.
\end{theoremA}

\begin{example}
\cref{fig:distributiveLattices} shows four distributive (hence semidistributive) interval hypergraphic lattices.
\cref{fig:semidistributiveLattices} shows two interval hypergraphic lattices which are semidistributive but not distributive.
\cref{fig:notSemidistributiveLattices} shows two interval hypergraphic lattices which are not semidistributive.
\end{example}

\begin{example}
The hypergraphic poset of~$\{ 1, 2, 3, 12, 23, 123 \}$ is a semidistributive (but not distributive) lattice (it is a pentagon).
In general, the hypergraphic poset of all intervals is the Tamari lattice, which is semidistributive but not distributive.
See \cref{fig:Tamari}.
\end{example}


\subsection{All join irreducible acyclic orientations}
\label{subsec:joinIrreducibles}

In this section, we assume that~$\II$ is an interval hypergraph closed under intersection so that the hypergraph poset~$P_\II$ is a lattice  by \cref{prop:latticeBackward}, and we describe all the join irreducible elements of~$P_\II$.
The following notations are illustrated in \cref{exm:Aij1,exm:Aij2,exm:Aij3,exm:Aij4,exm:Aij5,exm:Aij6} below.

\begin{notation}
For~$1 \le i < j \le n$ such that there exists~$I \in \II$ with~$\{i,j\} \subseteq I$, we  let
\[
J_{ij} =  \bigcap_{\substack{I \in \II \\ \{i,j\}\subseteq I}}  I
\]
(note that $J_{ij} \in \II$ since $\II$ is closed under intersection), and we set~$\mu_{ij} \eqdef \min(J_{ij})$ and~$\nu_{ij} \eqdef \max(J_{ij})$.
By definition~$\min(I) \le \mu_{ij} \le i < j \le \nu_{ij} \le \max(I)$ for any~$I \in \II$ such that~$\{i,j\} \subseteq I$.
\end{notation}

\begin{notation}
\label{not:cIJ}
Define~$\cIJ_\II$ to be the set of pairs~$(i,j)$ where~$1 \le i < j \le n$ are such that there exists~$I \in \II$ with~$\{i,j\} \subseteq I$ and $(i,j)$ satisfies the relation
\[
i = \max \Big( {[\mu_{ij}, j[} \; \ssm \!\!\!\! \bigcup_{\substack{J \in \II \\ J \subseteq {[\mu_{ij}, j[}}} \!\!\!\!\big( J \ssm \{\min(J)\}\big) \Big).
\]
\end{notation}

\begin{notation}
\label{not:joinIrreducibles}
For~$(i,j) \in \cIJ_\II$, consider the acyclic orientation~$A_{ij} \eqdef \Or_{(\mu_{ij}, \mu_{ij}+1, \ldots, j)}$, obtained as the image by the surjection map~$\Or$ of \cref{def:surjection} of the cycle permutation
\[
(\mu_{ij}, \mu_{ij}+1, \dots, j) = 12 \cdots (\mu_{ij}-1)j\,\mu_{ij} \cdots (j-1) (j+1) \cdots n
\]
obtained from the identity permutation by placing~$j$ just before $\mu_{ij}$.
Note that that for all $J \in \II$
\[
A_{ij}(J) =
\begin{cases}
	j & \text{if } j \in J \text{ and } \min(J) \ge \mu_{ij},\\
	\min(J) & \text{otherwise.}
\end{cases}
\]
\end{notation}

\begin{remark}
In~\cref{not:muj,not:Aj}, we defined the interval~$J_j$, the index~$\mu_j$, and the join irreducible~$A_j$ for $j\in \cJ_\II$. The set 
\[
X_j={[\mu_{j}, j[} \ssm \!\!\!\! \bigcup_{\substack{J \in \II \\ J \subseteq {[\mu_{j}, j[}}} \!\!\!\!\big( J \ssm \{\min(J)\}\big)
\]
is nonempty as it contains at least $\mu_j$. Let $\mu_i \le i_j \eqdef \max X_j$. Then we have
\[
\mu_{i_jj}=\mu_j
\qquad\text{ and }\qquad
J_{i_jj}=J_j.
\]
Furthermore, the pair $(i_j,j)\in\cIJ_\II$ since it satisfies all the conditions of \cref{not:cIJ}.
Comparing~\cref{not:Aj} and~\cref{not:joinIrreducibles} we see that $A_j=A_{i_jj}$.
Moreover, for all $(k,j)\in \cIJ_\II$, we must have $k\le i_j$.
See~\cref{exm:Aij1,exm:Aij2,exm:Aij3,exm:Aij4,exm:Aij5,exm:Aij6} below.
\end{remark}

\begin{notation}
For $(i,j), (k,\ell) \in \cIJ_\II$, we write $(i,j) \preccurlyeq (k,\ell)$ if and only $j \le \ell$ and~$\ell \in J_{ij}$ and~$\mu_{ij} \ge \mu_{k\ell}$.
Note in particular that~$i \ge k$ and~$j = \ell$ implies $(i,j) \preccurlyeq (k,\ell)$.
\end{notation}

\begin{example}
\label{exm:Aij1}
For the hypergraph~$\II \eqdef \{1, 2, 3, 12, 23, 123\}$ of \cref{exm:Tamari1,exm:oneMoreIrreducible,fig:Tamari}\,(left), we have~$\cIJ_\II = \{(1,2), (2,3), (1,3)\}$, and the join irreducible acyclic orientations~are
\[
A_{12} = A_2 = \acyclicOrientation{3}{1/2/2,2/3/2,1/3/2}
\qquad\quad
A_{23} = A_3 = \acyclicOrientation{3}{1/2/1,2/3/3,1/3/1}
\qquad\text{and}\qquad
A_{13} = \acyclicOrientation{3}{1/2/1,2/3/3,1/3/3}
\]
and we have~$A_{23} < A_{13}$.
Note that~$A_{12} = A_2$, $A_{23} = A_3$ from \cref{subsec:someJoinIrreducibles}, while~$A_{13} = A$ is the join irreducible acyclic orientation not covered by the description of \cref{subsec:someJoinIrreducibles}.
\end{example}

\begin{example}
\label{exm:Aij2}
For the hypergraph~$\II \eqdef \{1, 2, 3, 12, 23, 34, 123, 234, 1234\}$ of \cref{exm:Tamari2,fig:Tamari}\,(right), we have~$\cIJ_\II = \{(1,2), (2,3), (3,4), (1,3), (2,4), (1,4)\}$, and the join irreducible acyclic orientations~are
\begin{gather*}
A_{12} = \raisebox{-.4cm}{\acyclicOrientation{4}{1/2/2,2/3/2,3/4/3,1/3/2,2/4/2,1/4/2}}
\qquad\quad
A_{23} = \raisebox{-.4cm}{\acyclicOrientation{4}{1/2/1,2/3/3,3/4/3,1/3/1,2/4/3,1/4/1}}
\qquad\quad
A_{34} = \raisebox{-.4cm}{\acyclicOrientation{4}{1/2/1,2/3/2,3/4/4,1/3/1,2/4/2,1/4/1}}
\\[.3cm]
A_{13} = \raisebox{-.4cm}{\acyclicOrientation{4}{1/2/1,2/3/3,3/4/3,1/3/3,2/4/3,1/4/3}}
\qquad\quad
A_{24} = \raisebox{-.4cm}{\acyclicOrientation{4}{1/2/1,2/3/2,3/4/4,1/3/1,2/4/4,1/4/1}}
\qquad\quad
A_{14} = \raisebox{-.4cm}{\acyclicOrientation{4}{1/2/1,2/3/2,3/4/4,1/3/1,2/4/4,1/4/4}}
\end{gather*}
and we have~$A_{23} < A_{13}$ and~$A_{34} < A_{24} < A_{14}$.
\end{example}

\begin{example}
\label{exm:Aij3}
For the hypergraph~$\II \eqdef \{1, 2, 3, 4, 123, 23, 234, 1234\}$ illustrated in \cref{fig:semidistributiveLattices}\,(left), we have~$\cIJ_\II = \{(1,2), (2,3), (2,4), (1,4)\}$, and the join irreducible acyclic orientations~are
\begin{gather*}
A_{12} = \raisebox{-.2cm}{\acyclicOrientation{4}{1/3/2,2/3/2,2/4/2,1/4/2}}
\qquad
A_{23} = \raisebox{-.2cm}{\acyclicOrientation{4}{1/3/1,2/3/3,2/4/3,1/4/1}}
\qquad
A_{24} = \raisebox{-.2cm}{\acyclicOrientation{4}{1/3/1,2/3/2,2/4/4,1/4/1}}
\qquad
A_{14} = \raisebox{-.2cm}{\acyclicOrientation{4}{1/3/1,2/3/2,2/4/4,1/4/4}}
\end{gather*}
and we have~$A_{24} < A_{14}$.
\end{example}

\begin{example}
\label{exm:Aij4}
For the hypergraph~$\II \eqdef \{1, 2, 3, 4, 23, 34, 1234\}$ illustrated in \cref{fig:semidistributiveLattices}\,(right), we have~$\cIJ_\II = \{(1,2), (2,3), (3,4), (2,4)\}$, and the join irreducible acyclic orientations~are
\begin{gather*}
A_{12} = \raisebox{-.2cm}{\acyclicOrientation{4}{2/3/2,3/4/3,1/4/2}}
\qquad
A_{23} = \raisebox{-.2cm}{\acyclicOrientation{4}{2/3/3,3/4/3,1/4/1}}
\qquad
A_{34} = \raisebox{-.2cm}{\acyclicOrientation{4}{2/3/2,3/4/4,1/4/1}}
\qquad
A_{24} = \raisebox{-.2cm}{\acyclicOrientation{4}{2/3/2,3/4/4,1/4/4}}
\end{gather*}
and we have~$A_{12} < A_{24}$ and~$A_{34} < A_{24}$.
\end{example}

\begin{example}
\label{exm:Aij5}
For the hypergraph~$\II \eqdef \{1, 2, 3, 4, 12, 23, 34, 1234\}$ illustrated in \cref{fig:notSemidistributiveLattices}\,(left), we have~$\cIJ_\II = \{(1,2), (2,3), (3,4), (1,3), (1,4)\}$, and the join irreducible acyclic orientations~are
\begin{gather*}
A_{12} = \raisebox{-.2cm}{\acyclicOrientation{4}{1/2/2,2/3/2,3/4/3,1/4/2}}
\qquad\quad
A_{23} = \raisebox{-.2cm}{\acyclicOrientation{4}{1/2/1,2/3/3,3/4/3,1/4/1}}
\qquad\quad
A_{34} = \raisebox{-.2cm}{\acyclicOrientation{4}{1/2/1,2/3/2,3/4/4,1/4/1}}
\\
A_{13} = \raisebox{-.2cm}{\acyclicOrientation{4}{1/2/1,2/3/3,3/4/3,1/4/3}}
\qquad\quad
A_{14} = \raisebox{-.2cm}{\acyclicOrientation{4}{1/2/1,2/3/2,3/4/4,1/4/4}}
\end{gather*}
and we have~$A_{23} < A_{13}$ and~$A_{34} < A_{14}$.
\end{example}

\begin{example}
\label{exm:Aij6}
\enlargethispage{.4cm}
For the hypergraph~$\II \eqdef \{1, 2, 3, 4, 12, 23, 234, 1234\}$ illustrated in \cref{fig:notSemidistributiveLattices}\,(right), we have~$\cIJ_\II = \{(1,2), (2,3), (3,4), (1,3), (1,4)\}$, and the join irreducible acyclic orientations~are
\begin{gather*}
A_{12} = \raisebox{-.2cm}{\acyclicOrientation{4}{1/2/2,2/3/2,2/4/2,1/4/2}}
\qquad\quad
A_{23} = \raisebox{-.2cm}{\acyclicOrientation{4}{1/2/1,2/3/3,2/4/3,1/4/1}}
\qquad\quad
A_{13} = \raisebox{-.2cm}{\acyclicOrientation{4}{1/2/1,2/3/3,2/4/3,1/4/3}}
\\
A_{24} = \raisebox{-.2cm}{\acyclicOrientation{4}{1/2/1,2/3/2,2/4/4,1/4/1}}
\qquad\quad
A_{14} = \raisebox{-.2cm}{\acyclicOrientation{4}{1/2/1,2/3/2,2/4/4,1/4/4}}
\end{gather*}
and we have~$A_{23} < A_{13}$ and~$A_{24} < A_{14}$.
\end{example}

\pagebreak
\begin{lemma}
\label{lem:distinctIrreducibles2}
For $(i,j) \ne (k,\ell) \in \cIJ_\II$, we have $A_{ij} \ne A_{k\ell}$.
\end{lemma}

\begin{proof}
Assume that~$A_{ij} = A_{k\ell}$.
As~$j = A_{ij}(J_{ij}) = A_{k\ell}(J_{ij}) \in \{\ell, \mu_{ij}\}$ and~$\mu_{ij} < j$, we obtain that~$j = \ell = A_{k\ell}(J_{ij})$ hence that~$\mu_{ij} \ge \mu_{k\ell}$.
By symmetry, we have~$\mu_{ij} = \mu_{k\ell}$.
As $j = \ell$ and $\mu_{ij} = \mu_{k\ell}$, we obtain that~$i = k$ by definition of~$\cIJ_\II$.
We conclude that~$(i,j) = (k,\ell)$.
\end{proof}

\begin{lemma}
\label{lem:irrorder2}
For $(i,j), (k,\ell) \in \cIJ_\II$, we have $A_{ij} \le A_{k\ell} \iff (i,j) \preccurlyeq (k,\ell)$.
In particular~$j = \ell$ implies that $A_{ij}$ and~$A_{k\ell}$ are comparable in~$P_\II$.
\end{lemma}

\begin{proof}
By \cref{prop:sourceOrderI} we have~$A_{ij}\le A_{k\ell}$ if and only if~$A_{ij}(I) \le A_{k\ell}(I)$ for all~$I \in \II$.

For the forward direction, assume that~$A_{ij} \le A_{k\ell}$.
As~$j = A_{ij}(J_{ij}) \le A_{k\ell}(J_{ij}) \in \{\ell, \mu_{ij}\}$ and~$\mu_{ij} < j$, we obtain that~$j \le \ell$ and~$A_{k\ell}(J_{ij}) = \ell$, hence~$\ell \in J_{ij}$ and~$\mu_{ij} \ge \mu_{k\ell}$.

For the backward direction, assume that~$j \le \ell$, $\ell \in J_{ij}$ and~$\mu_{ij} \ge \mu_{k\ell}$.
Assume by means of contradiction that there is~$I \in \II$ such that~$A_{ij}(I) > A_{k\ell}(I)$.
As $A_{ij}(I) \in \{j, \min(I)\}$ and~$A_{k\ell} \in \{\ell, \min(I)\}$, and~$j \le \ell$, we must have~$A_{ij}(I) = j$ and~$A_{k\ell}(I) = \min(I)$.
As~$A_{ij}(I) = j$, we have~$j \in I$ and~$\mu_{ij} \le \min(I)$, hence~$\{i,j\} \subseteq [\mu_{i,j},j] \subseteq I$ so that~$\ell \in J_{ij} \subseteq I$ by minimality of~$J_{ij}$.
As~$A_{k\ell}(I) \ne \ell$ and~$\ell \in I$, we obtain that~$\min(I) < \mu_{k\ell}$.
We thus obtain that~$\mu_{ij} \le \min(I) < \mu_{k\ell}$ contradicting our assumption that~$\mu_{ij} \ge \mu_{k\ell}$.
\end{proof}

\begin{proposition}
\label{prop:AijJoinIrreducible}
If~$\II$ is an interval hypergraph closed under intersection, then the acyclic orientation~$A_{ij}$ is join irreducible for any~$(i,j) \in \cIJ_\II$.
\end{proposition}

\begin{proof}
As~$A_{ij} \ne \min(P_\II)$, we just need to prove that it covers a single acyclic orientation in~$P_\II$.
We consider a cover relation~\flip{A}{k}{\ell}{A_{ij}} in~$P_\II$ and prove that~$i = k$ and~$j = \ell$.

By definition, there is~$I \in \II$ such that~$\{k,\ell\} \subseteq I$ and~$A(I) = k$ while~$A_{ij}(I) = \ell$.
Since $A_{ij}(I) \in \{j, \min(I)\}$ and~$A_{ij}(I) = \ell > k \ge \min(I)$, we obtain that~$j = \ell$.

Assume now that~$i < k$.
Then there is~$J \in \II$ with~$J \subseteq [\mu_{ij}, j[$ and~$k \in J \ssm \{\min(J)\}$.
We obtain that~$A(J) = A_{ij}(J) = \min(J) \in [\mu_{ij}, j[ \subseteq I$ and~$A(I) = k \in J$ and~$A(J) \ne A(I)$, contradicting the acyclicity of~$A$.

Assume now that~$k < i$.
Since~$\flip{A}{k}{\ell}{A_{ij}}$ is a cover relation, the decreasing analogue of \cref{prop:isCoverI} implies that
\[
{]k,\ell[} \subseteq \!\!\!\! \bigcup_{\substack{J \in \II \\ A_{ij}(J) \in [k,\ell[}} \!\!\!\! \big( J \ssm \{A_{ij}(J)\} \big).
\]
As~$i \in {]k,\ell[}$, there is~$J \in \II$ with~$A_{ij}(J) \in {[k,\ell[} = {[k,j[}$ and~$i \in J \ssm \{A_{ij}(J)\}$.
As~${A_{ij}(J) \in \{j, \min(J)\}}$ by definition of~$A_{ij}$, we obtain that~$A_{ij}(J) = \min(J)$.
If $j \le \max(J)$, then $A_{ij}(I) = \ell \in J$ and~$A_{ij}(J) \in {[k,\ell[} \subseteq I$ and~$A_{ij}(I) \ne A_{ij}(J)$ contradict the acyclicity of~$A_{ij}$.
As~$\II$ is closed under intersection, we can thus assume that~$J \subseteq [\mu_{ij}, j[$, up to intersecting~$J$ with~$J_{ij}$.
We conclude that~$J \subseteq [\mu_{ij}, j[$ and~$i \in J \ssm \{A_{ij}(J)\} = J \ssm \{\min(J)\}$, contradicting the definition of~$i$ in \cref{not:cIJ}.
\end{proof}

\begin{proposition}
\label{prop:joinIrreducibles}
If~$\II$ is an interval hypergraph closed under intersection, then the map~${(i,j) \mapsto A_{ij}}$ is a bijection from~$\cIJ_\II$ to the join irreducible acyclic orientations of~$\II$.
\end{proposition}

\begin{proof}
\enlargethispage{.4cm}
Consider a join irreducible acyclic orientation~$B$ of~$\II$ and let~$\flip{A}{i}{j}{B}$ be the only increasing flip which ends at~$B$.
Let~$I \in \II$ be such that~$A(I) = i$ and~$B(I) = j$.
Observe that~$B(J) \in \{\min(J), j\}$ for any~$J \in \II$ by \cref{prop:alwaysFlippableI}.

We first prove that~$(i,j) \in \cIJ_\II$.
Let
\[
X \eqdef {[\mu_{ij},j[} \ssm \!\!\!\! \bigcup_{\substack{J \in \II \\ J \subseteq {[\mu_{ij}, j[}}} \!\!\!\! \big( J \ssm \{\min(J)\} \big).
\]
Since~$B(I) = j$ and~$\min(I) \le \mu_{i,j}$, \cref{prop:isFlipI} ensures that there is no~$J \in \II$ such that~${i \in J \ssm \{B(J)\}}$ and~$B(J) \in {[\mu_{i,j}, j[}$, which yields~$i \in X$ (using that~$B(J) \in \{\min(J),j\}$).
Moreover, for any~$k \in {]i,j[}$, \cref{prop:isCoverI} ensures that there is~$J \in \II$ such that~$B(J) \in {[i,j[}$ and~$k \in J \ssm \{B(J)\}$.
Since~${[i,j[} \subseteq {[\mu_{ij},j[}$ and~$B(J) = \min(J)$, we obtain that~$k \notin X$.
We conclude that~$i = \max(X)$, so that~$(i,j) \in \cIJ_\II$.

We now prove that~$B = A_{ij}$.
As already observed, we have~$B(J) \in \{\min(J), j\}$ for any~${J \in \II}$.
Assume first that there is~$J \in \II$ such that~$j \in J$ and~$\min(J) \ge \mu_{ij}$, but~$B(J) \ne j$.
Then ${B(J) = \min(J) \in {[\mu_{ij}, j[} \subseteq I \ssm \{B(I)\}}$ and~$B(I) = j \in J$, which contradicts the acyclicity of~$B$.
Assume now that there is~$J \in \II$ such that~$\min(J) < \mu_{ij}$ but~${B(J) = j}$, and consider such a~$J$ with~$\min(J)$ minimal.
For any~$K \in \II$ with~$A(K) \in {[\min(J),j[}$, we have~$j \notin K$ by acyclicity of~$B$, so that~${A(K) = \min(K) \ge \min(J)}$, thus~${\min(J) \notin K \ssm \{\min(K)\}}$.
\cref{prop:isFlipI} thus ensures that the orientation~$C$ obtained from~$B$ by flipping~$j$ to~$\min(J)$ is acyclic.
Moreover, as~$\min(J) < \mu_{ij}$, we have~$\min(J) \notin J_{ij}$, so that~$C(J_{ij}) = B(J_{ij}) = j > i = A(J_{ij})$.
We obtain that~$C < B$ but~$C \not\le A$ contradicting our assumption that~$B$ is join irreducible.
We conclude that~$B(J) = j$ if~$j \in J$ and~$\min(J) \ge \mu_{ij}$, and~$B(J) = \min(J)$ otherwise, so that~$B = A_{ij}$.
\end{proof}


\subsection{If~$\II$ is semidistributive then $P_\II$ is semidistributive}
\label{subsec:semidistributiveLatticeBackward}

We now prove the backward direction of \cref{thm:semidistributiveLatticeI}.

\begin{proposition}
\label{prop:semidistributiveBackward}
If~$\II$ is a join (resp.~meet) semidistributive interval hypergraph, then $P_\II$ is a join (resp.~meet) semidistributive lattice.
\end{proposition}

\begin{proof}
We prove the result for join semidistributivity, the result for meet semidistributivity follows by the symmetry of~\cref{prop:antiIsomorphism}.

Assume by means of contradiction that~$\II$ is a semidistributive interval hypergraph for which~$P_\II$ is not join semidistributive.
By \cref{rem:semidistributiveCriterion}, there exist a cover relation~$\flip{A}{p}{q}{B}$ of~$P_\II$ and some~$(i,j), (k,\ell) \in \cIJ_\II$ such that~$A_{ij}$ and~$A_{k\ell}$ are both minimal in~$P_\II$ with~$A \join A_{ij} = B = A \join A_{k\ell}$.
Assume by symmetry that~$j \le \ell$.
As~$A_{ij} \not\le A_{k\ell}$, we obtain by \cref{lem:irrorder2} that~$j < \ell$, and that~$\ell \notin J_{ij}$ or~$\mu_{ij} < \mu_{k\ell}$.

As~$A_{ij} \not\le A$ but~$A_{ij} \le B$, there is~$U \eqdef [u,u'] \in \II$ such that~$A(U) < A_{ij}(U) \le B(U)$ by \cref{prop:sourceOrderI}.
Since~$A$ and~$B$ only differ by the flip of~$p$ to~$q$, this implies that~${A(U) = p}$ while~${B(U) = q}$.
Moreover, by definition of~$A_{ij}$, we must have~$A_{ij}(U) = j$, hence~$j \in U$ and~${\mu_{ij} \le u}$.
We conclude that~$\mu_{ij} \le u \le p < j \le q \le u'$.
Similarly, ${\mu_{k\ell} \le u \le p < \ell \le q \le u'}$.

Observe that for any~$V \in \II$ such that~$A(V) \in {]p,q]}$, we have
\begin{itemize}
\item $p < \min(V)$ as otherwise~$A(U) = p \in V \ssm \{A(V)\}$ while~$A(V) \in {]p,q]} \subseteq U \ssm \{A(U)\}$ would contradict the acyclicity of~$A$,
\item if~$j \in V$, then~$A(V) = B(V) \ge A_{ij}(V) = j$.
\end{itemize}

Since~$\flip{A}{p}{q}{B}$ is a cover relation of~$P_\II$, \cref{prop:isCoverI} implies that there is~$S \eqdef [s,s'] \in \II$ such that~$A(S) \in {]p,q]}$ and~$j \in S \ssm \{A(S)\}$.
From the observation above, we have~$p < s$ and~$A(S) \ge j$.
We claim that we can moreover assume that~$s' < \ell$.
Indeed, suppose that~$s' \ge \ell$.
As~$\mu_{k\ell} \le p < s \le j < \ell \le s'$, we obtain that~$k < s \le j < \ell$.
Hence, there exists~$I \in \II$ such that~$j \in I \subseteq {[k,\ell[}$.
As~$\II$ is closed under intersection, we can replace~$S$ by~$S \cap I$ which proves the~claim.
Hence, we have found~$S \eqdef [s,s'] \in \II$ with~$p < s \le j \le A(S) \le s' < \ell \le q$, and we choose such an~$S$ so that~$s$ is minimal.

As~$A(S) \in S \subseteq {]p,q[}$, \cref{prop:isCoverI} implies that there is~$T \eqdef [t,t'] \in \II$ such that~${A(T) \in {]p,q]}}$ and~$A(S) \in T \ssm \{A(T)\}$.
We have~$A(T) \ge j$, as otherwise~${j \in [A(T), A(S)] \subseteq T}$ and~$A(T) < j$ would contradict the observation above.
We deduce that~$A(T) > s'$ as otherwise~$A(T) \in [j, s'] \subseteq S$ and~$A(S) \in T \ssm \{A(T)\}$ would contradict the acyclicity of~$A$.
Hence, we have found~$T \eqdef [t,t'] \in \II$ with~$p < t \le s' < t'$, and we choose such a~$T$ so that~$t'$ is minimal.

As~$\mu_{ij} \le p < s$ and~$j \in S$, we have~$i < s$, from which we deduce that there is~$R \eqdef [r,r'] \in \II$ such that~$R \subseteq {[\mu_{i,j},j [}$ and~$s \in R \ssm \{r\}$.
As~$r' < j \le s'$, we have~$r < s \le r' < s'$, and we choose such an~$R$ so that~$r$ is minimal.
Assume for a moment that~$t \le r$.
Then~$t < s$, which by minimality of~$s$ in our choice of~$S$ implies that~$\ell \in T$.
If~$\ell \notin J_{ij}$, then~$J_{ij} \cap T$ contradicts the minimality of~$t'$ in our choice of~$T$.
Hence, $\ell \in J_{ij}$ which implies that~$\mu_{ij} < \mu_{k\ell}$.
We conclude that~$i < k < j < \ell$, as~$\mu_{k\ell} \le p < t$ and~$\ell \in T$.
As~$\II$ is closed under intersection, $i < k < j < \ell$ implies that~$(i,j) \in \cIJ_\II$.
As~$i < k$, we have~$A_{kj} < A_{ij} \le B$.
Moreover, $A_{kj}(J_{kj}) = j > \mu_{kj} = A(J_{kj})$, so that~$A \not\le A_{kj}$.
Hence, $A_{kj} \join A = B$ and~$A_{kj} < A_{ij}$ contradicting the minimality of~$A_{ij}$.
We conclude that~$r < t$.

We have thus found~$R \eqdef [r,r']$, $S \eqdef [s,s']$, $T \eqdef [t,t']$ and~$U \eqdef [u,u']$ in~$\II$ with~$r < s \le r' < s'$, $r < t \le s' < t'$, and~$u \le p < \min(s,t)$ and~$s' < q \le u'$.
As~$\II$ is join semidistributive, there is~$V \eqdef [v,v'] \in \II$ such that~$v < s$ and~$s' < v' < t'$.
As~$\II$ is closed under intersection, it contains~$T \cap V = [\max(v,t), v']$.
But~$\max(v,t) \le s' < v' < t'$ which contradicts the minimality of~$t'$ in our choice of~$T$.
\end{proof}


\subsection{If~$P_\II$ is semidistributive then~$\II$ is semidistributive}
\label{subsec:semidistributiveLatticeForward}

We now prove the forward direction of \cref{thm:semidistributiveLatticeI}.

\begin{proposition}
\label{prop:semidistributiveForward}
If~$\II$ is an interval hypergraph such that $P_\II$ is a join (resp.~meet) semidistributive lattice, then~$\II$ is join (resp.~meet) semidistributive.
\end{proposition}

\begin{proof}
We prove the result for join semidistributivity, the result for meet semidistributivity follows by the symmetry of~\cref{prop:antiIsomorphism}.

First, we can assume that~$\II$ is closed under intersection, otherwise $P_\II$ is not even a lattice by \cref{prop:latticeForward}.
Assume now that there is~$[r,r'], [s,s'], [t,t'], [u,u'] \in \II$ such that $r < s \le r' < s'$, $r < t \le s' < t'$, $u < \min(s, t)$ and~$s' < u'$, and there is no~$[v,v'] \in \II$ such that~$v < s$ and~${s' < v' < t'}$.
Let
\[
j \eqdef \max \Big( [u,t'[ \ssm \!\!\!\! \bigcup_{\substack{J \in \II \\ J \subseteq [u,t'[}} \!\!\!\! \big(J \ssm \{\min(J)\} \big)\Big)
\quad\text{and}\quad
i \eqdef \mu_{jt'} \eqdef \max\set{\min(I)}{\{j,t'\} \subseteq I \in \II}.
\]
Note that $i \le j < s < s' < t'$.
We further assume below that~$i < j$, the argument for~$i = j$ is similar and even slightly simpler.

Consider the three permutations
\begin{align*}
\pi_A & \eqdef 1 \cdots (i-1) \, j \, i \cdots (j-1) \, (j+1) \cdots (s-1) \, t' \, s' \, s \cdots (s'-1) \, (s'+1) \cdots (t'-1) \, (t'+1) \cdots n \\
\pi_B & \eqdef 1 \cdots (i-1) \, s' \, i \cdots (s'-1) \, (s'+1) \cdots n \\
\pi_C & \eqdef 1 \cdots (i-1) \, t' \, i \cdots (t'-1) \, (t'+1) \cdots n
\end{align*}
(written in one line notation), and let~$A \eqdef \Or_{\pi_A}$, $B \eqdef \Or_{\pi_B}$ and $C \eqdef \Or_{\pi_C}$ be the corresponding acyclic orientations of~$\II$.

\para{Claim 1:} $A \join B = A \join C$.
Note that~$\pi_A = \projDown_A$.
Indeed, $\pi_A$ has only $3$ descents $ji$, $t's'$ and~$s's$.
Let~$\tau$, $\sigma$ and $\rho$ denote the permutations obtained from~$\pi_A$ by exchanging $ji$, $t's'$ and $s's$, respectively.
Then
\[
\Or_{\tau}(J_{jt'}) = i \ne j = X(J_{jt'}),
\;\;
\Or_{\sigma}([t,t']) = s' \ne t' = X([t,t'])
\;\;\text{and}\;\;
\Or_{\rho}([s,s']) = s \ne s' = X([s,s']).
\]
Hence, $\pi_A$ is indeed the minimal permutation for~$A$.
Similarly, $\pi_B = \projDown_B$ and $\pi_C = \projDown_C$.

We therefore obtain from \cref{prop:latticeBackward} that
$A \join B = \Or_{\pi_A \join \pi_B}$ and $A \join C = \Or_{\pi_A \join \pi_C}$.
As
\begin{align*}
\pi_A \join \pi_B & = 1 \cdots (i-1) \, t' \, s' \, j \, i \cdots (j-1) \, (j+1) \cdots (s'-1) \, (s'+1) \cdots (t'-1) \, (t'+1) \cdots n \\
\pi_A \join \pi_C & = 1 \cdots (i-1) \, t' \, j \, i \cdots (j-1) \, (j+1) \cdots (s-1) \, s' \, s \cdots (s'-1) \, (s'+1) \cdots (t'-1) \, (t'+1) \cdots n
\end{align*}
and there is no~$[v,v'] \in \II$ such that~$v < s$ and~$s' < v' < s'$ by assumption, we have~$A \join B = A \join C$.

\para{Claim 2:} $A \ge B \meet C$.
For any interval~$I$, we have
\[
B(I) = \begin{cases} s' & \text{if $s' \in I$ and~$i \le \min(I)$,} \\ \min(I) & \text{otherwise} \end{cases}
\quad\text{and}\quad
C(I) = \begin{cases} t' & \text{if $t' \in I$ and~$i \le \min(I)$,} \\ \min(I) & \text{otherwise} \end{cases}
\]
Hence, $(B \meet C)(I) = \min(I)$ except maybe if~$[s',t'] \subseteq I$ and~$i \le \min(I)$.
For such~$I$, we actually have that
\begin{itemize}
\item $(B \meet C)(I) \le j$ since $B(I) = s'$ and~$C(J) = \min(J)$ for all~$J \subseteq [u,t'[$ and using the definition of~$j$,
\item $A(I) = j$ since~$j \in I$ and~$i \le \min(I)$.
\end{itemize}
Hence, we obtain that~$A(I) \ge (B \meet C)(I)$ for all~$I \in \II$ so that~$A \ge B \meet C$ by \cref{prop:sourceOrderI}.

\para{Claim 3:} $A \join B = A \join C \ne A = A \join (B \meet C)$.
The first equality holds by Claim 1.
The last equality holds by Claim 2.
The inequality holds since~$A(J_{jt'}) = j$ while~$(B \meet C)(J_{jt'}) \ge C(J_{jt'}) = t'$.
We conclude that~$P_\II$ is not semidistributive.
\end{proof}

\pagebreak
\begin{example}
For the hypergraph~$\II \eqdef \{1, 2, 3, 4, 12, 23, 34, 1234\}$ illustrated in \cref{fig:notSemidistributiveLattices}\,(right), we have
\[
\begin{array}{c|c|c|c|c|c|c|c|c|c}
r & r' & s & s' & t & t' & u & u' & i & j \\
\hline
1 & 2 & 2 & 3 & 3 & 4 & 1 & 4 & 1 & 1
\end{array}
\]
and
\[
A = \raisebox{-.2cm}{\acyclicOrientation{4}{1/2/1,2/3/3,3/4/4,1/4/1}}
\qquad
B = \raisebox{-.2cm}{\acyclicOrientation{4}{1/2/1,2/3/3,3/4/3,1/4/3}}
\qquad
C = \raisebox{-.2cm}{\acyclicOrientation{4}{1/2/1,2/3/2,3/4/4,1/4/4}}
\]
so that
\[
A \join B = A \join C = \raisebox{-.2cm}{\acyclicOrientation{4}{1/2/1,2/3/3,3/4/4,1/4/1}} \ne C = C \join \raisebox{-.2cm}{\acyclicOrientation{4}{1/2/1,2/3/2,3/4/3,1/4/1}} = A \join (B \meet C).
\]
\end{example}

\begin{example}
For the hypergraph~$\II \eqdef \{1, 2, 3, 4, 12, 23, 234, 1234\}$ illustrated in \cref{fig:notSemidistributiveLattices}\,(right), we have
\[
\begin{array}{c|c|c|c|c|c|c|c|c|c}
r & r' & s & s' & t & t' & u & u' & i & j \\
\hline
1 & 2 & 2 & 3 & 2 & 4 & 1 & 4 & 1 & 1
\end{array}
\]
and
\[
A = \raisebox{-.2cm}{\acyclicOrientation{4}{1/2/1,2/3/3,2/4/4,1/4/1}}
\qquad
B = \raisebox{-.2cm}{\acyclicOrientation{4}{1/2/1,2/3/3,2/4/3,1/4/3}}
\qquad
C = \raisebox{-.2cm}{\acyclicOrientation{4}{1/2/1,2/3/2,2/4/4,1/4/4}}
\]
so that
\[
A \join B = A \join C = \raisebox{-.2cm}{\acyclicOrientation{4}{1/2/1,2/3/2,2/4/2,1/4/1}} \ne A = A \join \raisebox{-.2cm}{\acyclicOrientation{4}{1/2/1,2/3/3,2/4/4,1/4/1}} = A \join (B \meet C).
\]
\end{example}


\section{Quotient interval hypergraphic lattices}
\label{sec:quotient}

In this section, we prove \cref{thm:quotientLatticeI} which we first introduce properly:

\begin{definition}
\label{def:subinterval}
We say that an interval hypergraph~$\II$ is \defn{closed under initial} (resp.~\defn{final}) \defn{subintervals} if~$[i,k] \in \II$ implies~$[i,j] \in \II$ (resp.~$[j,k] \in \II$) for any~$1 \le i < j < k \le n$.
We say that~$\II$ is \defn{closed under subintervals} if it is closed under both initial and final subintervals, or equivalently if~$[i,\ell] \in \II$ implies~$[j,k] \in \II$ for any~$1 \le i \le j \le k \le \ell \le n$.
\end{definition}

\begin{remark}
Note that if an interval hypergraph~$\II$ containing all singletons~$\{i\}$ for~$i \in [n]$ is closed under initial (resp.~final) subintervals, then it is also closed under intersection.
Indeed, for any~$1 \le i<j < k < \ell \le n$ with~$[i,k] \in \II$ and~$[j,\ell] \in \II$, we have~$[j,k] \in \II$ since it is an initial (resp.~final) interval of~$[j,\ell]$ (resp.~of~$[i,k]$).
Hence, we do not need to add the closed by intersection condition.
\end{remark}

\begin{definition}
A map~$\phi$ between two meet semilattices~$(L, \le, \meet)$ and~$(L', \le', \meet')$ is a \defn{meet semilattice morphism} if~$\phi(a \meet b) = \phi(a) \meet' \phi(b)$ for all~$a,b \in L$.
Equivalently, if
\begin{itemize}
\item each fiber~$F$ of~$\phi$ is order convex (meaning that~$a < b < c$ and~$a,c \in F$ implies~$b \in F$) and admits a unique minimal element,
\item the map~$\projDown : L \to L$ sending an element~$a$ to the minimal element~$b$ with~$\phi(a) = \phi(b)$ is order preserving.
\end{itemize}
The join semilattice morphisms are defined and characterized dually.
\end{definition}

\begin{theoremA}
For an interval hypergraph $\II$ on~$[n]$, the poset morphism~$\Or$ from the weak order on~$\fS_n$ to the interval hypergraphic poset $P_\II$ is a meet (resp.~join) semilattice morphism if and only if $\II$ is closed under initial (resp.~final) subintervals.
\end{theoremA}

\begin{proof}
We only prove the meet semilattice morphism, the result for the join semilattice morphism follows by the symmetry of~\cref{prop:antiIsomorphism}.

Assume first that~$\II$ is not closed under initial subintervals, and let~$1 \le i < j < k \le n$ be such that~$[i,j] \notin \II$ while~$[i,k] \in \II$.
Note that we can assume that~$k = j+1$.
Let~$\omega \eqdef i j k X$, $\sigma \eqdef j i k X$, $\tau \eqdef k i j X$, and $\rho \eqdef k j i X$, where~$X$ is an arbitrary permutation of~$[n] \ssm \{i,j,k\}$.
Since~$\II$ is closed under intersection and contains~$[i,k]$ but not~$[i,j]$ and~$k = j+1$, any~$J \in \II$ containing~$\{i,j\}$ also contains~$k$.
Hence, we have~$\Or_\tau = \Or_\rho$, so that~$\projUp(\tau) \le \rho$.
Moreover, since~$[i,k] \in \II$, we have~${\Or_\omega < \Or_\sigma}$.
By \cref{prop:latticeBackward}, we obtain that~${\Or_\sigma \meet \Or_\tau = \Or_{\projUp(\sigma) \meet \projUp(\tau)} \ge \Or_{\sigma \meet \rho} = \Or_\sigma > \Or_\omega = \Or_{\sigma \meet \tau}}$.
Hence, $\Or$ is not a meet semilattice morphism.

Conversely, assume that~$\II$ is closed under initial subintervals.
We prove by induction on the length of~$\tau$ that~$\sigma \le \tau$ implies~$\projDown(\sigma) \le \projDown(\tau)$.
If~$\tau = \projDown(\tau)$, we have~$\projDown(\sigma) \le \sigma < \tau = \projDown(\tau)$.
If there is a permutation~$\rho$ such that~$\sigma < \rho < \tau$, then we have~$\projDown(\sigma) \le \projDown(\rho)$ by induction, so that we just need to show that~$\projDown(\rho) \le \projDown(\tau)$.
We can thus assume that~$\sigma \lessdot \tau$ is a cover relation, with~$\tau \ne \projDown(\tau)$.
Let~$\tau' \lessdot \tau$ be such that~$\Or_\tau = \Or_{\tau'}$.
Let~$p,q \in [n-1]$ be such that~$\sigma$ (resp.~$\tau'$) is obtained from~$\tau$ by exchanging its entries at positions~$p$ and~$p+1$ (resp.~$q$ and~$q+1$).
We distinguish three cases:

\para{If~$|p-q| > 1$.} Consider~$\sigma' = \sigma \meet \tau'$. As~$\Or_\tau = \Or_{\tau'}$, we have~$\min(\tau^{-1}(I)) \ne q$ or~$\tau(q+1) \notin I$ for any~$I \in \II$.
Since~$|p-q| > 1$, it implies that~$\min(\sigma^{-1}(I)) \ne q$ or~$\sigma(q+1) \notin I$.
Hence we obtain that~$\Or_\sigma = \Or_{\sigma'}$.
Since~$\sigma' < \tau' < \tau$, we obtain by induction that~$\projDown(\sigma') \le \projDown(\tau')$.
As~$\Or_\sigma = \Or_{\sigma'}$ and~$\Or_\tau = \Or_{\tau'}$, we conclude that~$\projDown(\sigma) \le \projDown(\tau)$.

\para{If~$p = q-1$.} Then~$\sigma = X j k i Y$, $\tau = X k j i Y$ and $\tau' = X k i j Y$ for some letters~$1 \le i < j < k \le n$ and some words~$X,Y$ on~$[n]$.
Let $\sigma' \eqdef X j i k Y$, $\sigma'' \eqdef X i j k Y$ and~$\tau'' \eqdef X i k j Y$.
If~$\Or_\sigma \ne \Or_{\sigma'}$, then there exist~$1 \le u \le i < k \le v \le n$ such that~$[u,v] \in \II$ and~$q = \min \big( \sigma^{-1}([u,v]) \big)$. 
As~$\II$ is closed under final subintervals, we have~$[u,j] \in \II$ and~$q = \min \big( \tau^{-1}([u,j]) \big)$, which contradicts that~$\Or_\tau = \Or_{\tau'}$.
If~$\Or_{\sigma'} \ne \Or_{\sigma''}$, then there exists~$1 \le u \le i < j \le v \le n$ such that~$[u,v] \in \II$ and~$p = \min \big( \sigma'^{-1}([u,v]) \big)$. 
This would imply that~$q = \big( \tau^{-1}([u,v]) \big)$, contradicting that~${\Or_\tau = \Or_{\tau'}}$.
Since~$\sigma'' < \tau' < \tau$, we obtain by induction that~$\projDown(\sigma'') \le \projDown(\tau')$.
As~$\Or_\sigma = \Or_{\sigma'} = \Or_{\sigma''}$ and~${\Or_\tau = \Or_{\tau'}}$, we conclude that~$\projDown(\sigma) \le \projDown(\tau)$.

\para{If~$p = q+1$.} Then~$\sigma = X k i j Y$, $\tau = X k j i Y$ and $\tau' = X j k i Y$ for some letters~${1 \le i < j < k \le n}$ and some words~$X,Y$ on~$[n]$.
Let~$\sigma' \eqdef X i k j Y$, $\sigma'' \eqdef X i j k Y$ and~$\tau'' \eqdef X j i k Y$. 
If~$\Or_\sigma \ne \Or_{\sigma'}$, then there exist~$1 \le u \le i < k \le v \le n$ such that~$[u,v] \in \II$ and~$q = \min \big( \sigma^{-1}([u,v]) \big)$.
This would imply that~$p = \min \big( \tau^{-1}([u,v]) \big)$, which contradicts that~$\Or_\tau = \Or_{\tau'}$.
If~$\Or_{\sigma'} \ne \Or_{\sigma''}$, then there exists~$i < u \le j < k \le v \le n$ such that~$[u,v] \in \II$ and~$p = \min \big( \sigma'^{-1}([u,v]) \big)$. 
This would imply that~$q = \big( \tau^{-1}([u,v]) \big)$, contradicting that~$\Or_\tau = \Or_{\tau'}$.
Since~$\sigma'' < \tau' < \tau$, we obtain by induction that~$\projDown(\sigma'') \le \projDown(\tau')$.
As~$\Or_\sigma = \Or_{\sigma'} = \Or_{\sigma''}$ and~$\Or_\tau = \Or_{\tau'}$, we conclude that~$\projDown(\sigma) \le \projDown(\tau)$.
\end{proof}

\begin{corollary}
\label{coro:quotientLatticeI}
For an interval hypergraph $\II$ on~$[n]$, the poset morphism~$\Or$ from the weak order on~$\fS_n$ to the interval hypergraphic poset $P_\II$ is a lattice morphism if and only if $\II$ is closed under subintervals.
In this case, $P_\II$ is a Cartesian product of Tamari lattices.
\end{corollary}

\begin{remark}
It follows from \cref{prop:surjectionFactorizeI} that the conditions of \cref{thm:quotientLatticeI,coro:quotientLatticeI} also characterize the interval hypergraphs which are (semi)lattice quotients of the Tamari lattice.
\end{remark}


\section*{Acknowledgments}

We are grateful to two anonymous referees for helpful suggestions.


\bibliographystyle{alpha}
\bibliography{IH_lattices}
\label{sec:biblio}


\end{document}